\documentclass[10pt]{amsproc}
\usepackage{amssymb,xy,graphicx}
\usepackage{color}
\xyoption{all}
\usepackage{hyperref}

\DeclareMathOperator{\End}{End}
\DeclareMathOperator{\Hom}{Hom}
\DeclareMathOperator{\intP}{\underline{\otimes}}
\DeclareMathOperator*{\Colim}{colim}
\DeclareMathOperator{\GL}{GL}
\newcommand{\Rep}{\mathrm{Rep\,}}

\newcommand{\cat}[1]{\mathcal{#1}}

\newcommand{\colim}[2]{\Colim_{{\bf h}_{#2}\to #1}}
\renewcommand{\mod}[1]{{#1}\text{-}\mathrm{Mod}}
\newcommand{\ev}{\mathrm{ev}}
\newcommand{\sgn}{\mathrm{sgn}}
\newcommand{\SG}{\mathbf S}
\newcommand{\AG}{\mathbf A}
\newcommand{\id}{\mathrm{id}}
\newcommand{\ZZ}{\mathbf Z}
\newcommand{\CC}{\mathbf C}
\newcommand{\sch}{\mathrm S}
\newcommand{\asch}{\mathrm{AS}}

\newtheorem{theorem}{Theorem}[section]
\newtheorem{lemma}[theorem]{Lemma}

\newtheorem{corollary}[theorem]{Corollary}
\theoremstyle{definition}
\newtheorem{definition}[theorem]{Definition}
\theoremstyle{remark}
\newtheorem{example}[theorem]{Example}
\newtheorem{remark}[theorem]{Remark}
\numberwithin{enumi}{theorem}
\renewcommand{\theenumi}{\thetheorem.\alph{enumi}}

\title[Alternating Schur Algebras and Koszul Duality]{Schur Algebras for the Alternating Group\\and Koszul Duality}
\author[Geetha]{Thangavelu Geetha}
\address{Indian Institute of Science Education and Research, Thiruvananthapuram}\email{geetha\_curie@yahoo.co.in}
\author[Prasad]{Amritanshu Prasad}
\address{The Institute of Mathematical Sciences, HBNI, Chennai}\email{amri@imsc.res.in}
\author[Srivastava]{Shraddha Srivastava}
\address{The Institute of Mathematical Sciences, HBNI, Chennai}\email{maths.shraddha@gmail.com}
\begin{document}

\begin{abstract}
  We introduce the alternating Schur algebra $\asch_F(n,d)$ as the commutant of the action of the alternating group $\AG_d$ on the $d$-fold tensor power of an $n$-dimensional $F$-vector space.
  When $F$ has characteristic different from $2$, we give a basis of $\asch_F(n,d)$ in terms of bipartite graphs, and a graphical interpretation of the structure constants.
  We introduce the abstract Koszul duality functor on modules for the even part of any $\ZZ/2\ZZ$-graded algebra.
  The algebra $\asch_F(n,d)$ is $\ZZ/2\ZZ$-graded, having the classical Schur algebra $\sch_F(n,d)$ as its even part.
  This leads to an approach to Koszul duality for $\sch_F(n,d)$-modules that is amenable to combinatorial methods.
  We characterize the category of $\asch_F(n,d)$-modules in terms of $\sch_F(n,d)$-modules and their Koszul duals.
  We use the graphical basis of $\asch_F(n,d)$ to study the dependence of the behavior of derived Koszul duality on $n$ and $d$.
\end{abstract}
\maketitle
\section{Introduction}
\label{sec:introduction}
\subsection{Schur-Weyl duality and its variants}
\label{sec:dualities}
Frobenius determined the irreducible characters of the symmetric group $\SG_{d}$ over $\CC$, the field of complex numbers, in 1900 \cite{frobenius1900}.
Building on this, Schur classified the irreducible polynomial representations of $\GL_n(\CC)$ and computed their characters in his PhD thesis \cite{schur_thesis}.
The group $\GL_n(\CC)$ acts on the factors of $(\CC^n)^{\otimes d}$, while $\SG_d$ permutes the tensor factors.
In 1927, Schur used these commuting actions to reprove the results of his dissertation \cite{Schur1927}.
Following Weyl's expositions of this method \cite{Weyl,Weyl-quantum}, it is known as Schur-Weyl duality.

Over the years, several variants of Schur-Weyl duality have emerged.
Shrinking $\GL_n(\CC)$ to the orthogonal group $\mathrm O_n(\CC)$, Brauer obtained the duality between Brauer algebras $\mathrm{Br}_d(n)$ and $\mathrm O_n(\CC)$ \cite{MR1503378}.
Motivated by the Potts model in statistical mechanics, Jones \cite{Jones} and Martin \cite{Martin} further shrunk $\mathrm O_n(\CC)$ down to $\SG_n$, obtaining the partition algebras $\mathrm P_d(n)$.
Bloss \cite{Bloss} reduced $\SG_n$ to $\AG_n$ to obtain an algebra $\mathrm{AP}_d(n)$ which coincides with the partition algebra when $n\geq 2d+2$.
We take the smallest possible step in the opposite direction: we reveal what takes the place of the polynomial representations of $\GL_n(\CC)$ when the action of the symmetric group $\SG_d$ is restricted to the alternating group $\AG_d$.
The situation is summarized in Table~\ref{tab:dualities}.
The significance of this investigation lies in its connection with the Koszul duality functor on the category of homogeneous polynomial representations of $\GL_n(\CC)$ of degree $d$.
\begin{table}
  \begin{displaymath}
    \begin{array}{|ccc|}
      \hline
      & {F^n}^{\otimes d} &\\
      \hline
      {??} & \text{\textbf{This article}} &  {\AG_d} \\
      \cup & & \cap\\
      {\GL_n(F)}& \text{Schur-Weyl} & {\SG_d}\\
      \cup & & \cap\\
      {\mathrm O_n(F)} & \text{Brauer} & {\mathrm{Br}_d(n)}\\
      \cup & & \cap\\
      {\SG_n} & \text{Martin and Jones} & {\mathrm P_d(n)}\\
      \cup & & \cap \\
      {\AG_{n}} & \text{Bloss}& {\mathrm{AP}_{d}(n)} \\
      \hline
    \end{array}
  \end{displaymath}
  \caption{Dualities arising from tensor space\label{tab:dualities}}
\end{table}
\subsection{Schur algebras for the alternating group}
\label{sec:schur-algebr-altern}
Motivated by Green \cite[Theorem~2.6c]{MR2349209}, define the Schur algebra as
\begin{displaymath}
  \sch_F(n,d) = \End_{\SG_d}((F^n)^{\otimes d})
\end{displaymath}
for any field $F$, and positive integers $n$ and $d$.
When $F$ is infinite, then $\sch_F(n,d)$-modules are the same as homogeneous polynomial representations of $\GL_n(F)$ of degree $d$ (see \cite[Section 2.4]{MR2349209} and \cite[Section 6.2]{rtcv}).
Define the alternating Schur algebra $\asch_F(n,d)$ by replacing $\SG_d$ by $\AG_d$ in the definition above:
\begin{displaymath}
  \asch_F(n,d) = \End_{\AG_d}((F^n)^{\otimes d}).
\end{displaymath}
When $F$ has characteristic different from $2$, this algebra has a decomposition (Lemma~\ref{lemma:z2-grading})
\begin{equation}
  \label{eq:intro}
  \asch_F(n,d)=\sch_F(n,d)\oplus \sch^-_{F}(n,d)
\end{equation}
as a $\mathbb{Z}/2\mathbb{Z}$-graded algebra.
Here $\sch^-_{F}(n,d)=\Hom_{\SG_d}((F^{n})^{\otimes d}, (F^{n})^{\otimes d}\otimes \sgn)$, where $\sgn$ denotes the sign character of $\SG_d$.
The subspace $\sch^-_F(n,d)$ is an\linebreak $(\sch_F(n,d),\sch_F(n,d))$-bimodule.

When $n^2<d$, then $\sch^-_F(n,d)=0$, and $\asch_F(n,d)=\sch_F(n,d)$, as observed by Regev~\cite[Theorem 1]{MR1898388}. But when $n^{2}\geq d$, $\sch^-_{F}(n,d)\neq 0$, and in Lemma~\ref{lm:tilting}, we note that $\sch^-_{F}(n,d)$ is a full tilting left $\sch_F(n,d)$-module as studied by Donkin in \cite[Section 3]{Donkin}.
\subsection{Bases and structure constants}
\label{sec:bases-struct-const}
Schur \cite{Schur1927} gave a combinatorial description of a basis and the corresponding structure constants of the Schur algebra (see also \cite[Section~2.3]{MR2349209}).
By indexing Schur's basis of $\sch_F(n,d)$ by bipartite multigraphs with $n+n$ vertices and $d$ edges, M\'endez~\cite{MR1897636} (see also Geetha and Prasad~\cite{structure}) gave a graphic interpretation of the structure constants.
We describe a basis of $\sch^-_F(n,d)$ in terms of bipartite simple graphs in Theorem~\ref{thm:basis_S-}.
So from the decomposition~\eqref{eq:intro}, a basis of $\asch_F(n,d)$ is obtained.
A graphic interpretation of the structure constants of $\asch_F(n,d)$ is given in Theorems~\ref{theorem:schur-structure-const} and \ref{theorem:stru-consts-alt}.
This will be used to derive properties of $\asch_F(n,d)$, its bimodule $\sch^-_F(n,d)$, and Koszul duality.
\subsection{Koszul duality and modules}
\label{sec:kosz-dual-modul}
The term \emph{Koszul duality} is used for several constructions which interchange the roles of exterior and symmetric powers.

The earliest notion of Koszul duality was introduced by Priddy \cite{MR0265437}.
It applies to \emph{pre-Koszul algebras}, which are also called \emph{quadratic algebras}.
A pre-Koszul algebra is a quotient of a tensor algebra $T(V)=\bigoplus_{n\geq 0}\otimes^n V$ by a two-sided ideal $I$ that is generated in degree two.
Its Koszul dual is the algebra $T(V^*)/(I\cap (V\otimes V))^\perp$; the quotient of the dual tensor algebra by the annihilator in degree two of $I$.
In this setting the Koszul dual of the symmetric algebra of $V$ is the exterior algebra of $V^*$.

Bernstein, Gelfand, and Gelfand \cite[Theorem 3]{MR509387} introduced an equivalence between the bounded derived categories of graded modules over symmetric and exterior algebras, which was called the  Koszul duality functor by Beilinson, Ginsburg, and Schectman \cite{MR1048505}.

Friedlander and Suslin~\cite{MR1427618} introduced the category of \emph{strict polynomial functors} of degree $d$ as the representations of the Schur category of degree $d$, for each non-negative integer $d$ (see Section~\ref{sec:strict-polyn-funct}).
The category of strict polynomial functors of degree $d$ unifies the categories of homogeneous polynomial representations of $\GL_n(F)$ of degree $d$ across all $n$.
Standard examples of strict polynomial functors of degree $d$ are the $d$th tensor power functor $\otimes^{d}$, the $d$th symmetric power functor $\text{Sym}^{d}$, and the $d$th exterior power functor $\wedge^{d}$.
Evaluating a strict polynomial functor of degree $d$ at $F^n$ gives an $\sch_F(n,d)$-module for each $n$.
Friedlander and Suslin showed that this evaluation functor is an equivalence of categories when $n\geq d$.
Cha\l{}upnik \cite{MR2414328} and Touz\'e \cite{MR3331175} used the term Koszul duality to refer to a functor on the category of strict polynomial functors of degree $d$ which takes the Schur functor associated to the partition $\lambda$ of $d$ to the Weyl functor associated with the partition $\lambda'$ \emph{conjugate} to $\lambda$.
Krause \cite{MR3077659} discovered an internal tensor product on the category of strict polynomial functors of fixed degree $d$.
Given such a tensor product it was then natural for him to define Koszul duality in this category as tensor product with $\wedge^d$.
This definition is different from the Koszul duality functors defined earlier by Cha\l{}upnik and Touz\'e.
Those coincide with a duality defined by Ringel \cite{MR1128706} using tilting modules for quasi-hereditary algebras.
This tilting module was described by Donkin \cite{Donkin} in the case of Schur algebras.

We introduce the term \emph{abstract Koszul duality} to refer to a very simple functor which makes sense for any $\ZZ/2\ZZ$-graded algebra $\asch = \sch\oplus \sch^-$.
The abstract Koszul dual of an $\sch$-module $V$ is defined as
\begin{displaymath}
  D(V) = \sch^-\otimes_{\sch} V.
\end{displaymath}
The multiplication operation on $\asch$ gives rise to an $(\sch,\sch)$-bimodule homomorphism $\phi:\sch^-\otimes_{\sch} \sch^-\to S$ and hence a natural transformation from $D\circ D$ to the identity functor on the category of $\sch$-modules.
We prove (Theorem \ref{thm:iso}) that the category of $\asch$-modules is same as the category of pairs $(M,\theta_{M})$ where $M$ is an $\sch$-module and $\theta_{M}:D(M)\to M$ is compatible with $\phi$ in the sense of \eqref{eq:2}.

In Section~\ref{sec:koszul-duality-schur}, we specialize to the case $\asch_F(n,d)=\sch_F(n,d)\oplus \sch^-_F(n,d)$ to obtain a Koszul duality functor $D$ on the category of $\sch_F(n,d)$-modules.
In Theorem \ref{thm:duals_comapre}, we show that the evaluation at $F^n$ of the Koszul duality functor of Krause is naturally isomorphic to our Koszul duality functor when $n\geq d$.
In this sense, our abstract Koszul duality functor on Schur algebras coincides with Krause's Koszul duality.

Our description of the structure constants of $\asch_F(n,d)$ allow us to give a direct combinatorial proof of the well-known fact that, when $n\geq d$ and when the characteristic of $F$ is $0$ or greater than $d$, then abstract Koszul duality is an equivalence (Theorem \ref{theorem:duality-is-equiv}).

Krause~\cite{MR3077659} showed that derived Koszul duality functor is an auto-equivalence of the unbounded derived category of strict polynomial functors.
Since the evaluation functor is an equivalence, this implies that derived Koszul duality is an auto-equivalence at the level of unbounded derived category of $\sch_F(n,d)$-modules when $n\geq d$.
However, this does not address the case where $n<d$.
Using our combinatorial methods, we show that derived Koszul duality is not an equivalence when $n<d$ (Theorem~\ref{prop:deri_kos}).
This proof uses a criterion of Happel \cite{Happel} for a tensor functor to be a derived equivalence.
In the context of derived Koszul duality, this criterion requires that the canonical algebra homomorphism $\sch_F(n,d)\to \End_{\sch_F(n,d)}(\sch^-_F(n,d))$ is an isomorphism.
Donkin \cite[Proposition 3.7]{Donkin} proved this when $n\geq d$.
When the characteristic of $F$ is not $2$ we give a combinatorial proof of Donkin's result, and also show that it fails when $n<d$ (Theorem~\ref{thm:iso_donkin}).
Figure~\ref{fig:kd} on page~\pageref{fig:kd} describes the behavior of Koszul duality for all values of the parameters $n$ and $d$.

We conclude this paper by discussing a possible application of our techniques to Bloss's alternating partition algebra, and a diagrammatic interpretation of the Schur category (Section~\ref{sec:concluding-remarks}).
\section{The alternating Schur algebra}
\label{sec:altern-schur-algebra}
Let $F$ be a field of characteristic different from $2$, $n$ and $d$ be positive integers.
The symmetric group $\SG_d$ acts on the tensor space $(F^n)^{\otimes d}$ by permuting the tensor factors.
The Schur algebra can be defined as
\begin{displaymath}
  \sch_F(n,d) := \End_{\SG_d} ((F^n)^{\otimes d}).
\end{displaymath}
By restricting the action of $\SG_d$ to the alternating group $\AG_d$, define the \emph{alternating Schur algebra} as
\begin{displaymath}
  \asch_F(n,d) := \End_{\AG_d} ((F^n)^{\otimes d}).
\end{displaymath}
Clearly, $\sch_F(n,d)$ is a subalgebra of $\asch_F(n,d)$.
\begin{lemma}
  \label{lemma:z2-grading}
  For any representations $V$ and $W$ of $\SG_d$,
  \begin{equation}
    \label{eq:Alt-end}
    \Hom_{\AG_d}(V,W)=\Hom_{\SG_d}(V,W)\oplus \Hom_{\SG_d}(V, W\otimes \sgn).
  \end{equation}
  Here $W\otimes \sgn$ denotes the twist of $W$ by the sign character $\sgn:\SG_d\to \{\pm 1\}$.
\end{lemma}
Define,
\begin{equation}
  \label{eq:S-}
  \sch^-_F(n,d): = \Hom_{\SG_d}((F^n)^{\otimes d}, (F^n)^{\otimes d}\otimes \sgn).
\end{equation}
Lemma~\ref{lemma:z2-grading} gives a $\mathbf Z/2\mathbf Z$-grading of $\asch_F(n,d)$ in the sense of Bourbaki \cite[Chapter III, Section~3.1]{MR1727844}:
\begin{equation}
  \label{eq:z2-grading}
  \asch_F(n,d) = \sch_F(n,d)\oplus \sch_F^-(n,d).
\end{equation}
The summand $\sch^-_F(n,d)$ is an $(\sch_F(n,d),\sch_F(n,d))$-bimodule.
Recall that a \emph{weak composition} of $d$ with $n$ parts is a vector $\lambda=(\lambda_1,\dotsc,\lambda_n)$ of non-negative integers summing to $d$.
Let $\Lambda(n,d)$ denote the set of weak compositions of $d$ with $n$ parts.
For each $\lambda=(\lambda_1,\dotsc,\lambda_n)\in \Lambda(n,d)$, define
\begin{displaymath}
  \wedge^\lambda F^n = \wedge^{\lambda_1}F^n\otimes \dotsb \otimes \wedge^{\lambda_n}F^n.
\end{displaymath}
where, for a non-negative integer $s$,
$\wedge^sF^n$ is the $s$th exterior power of $F^n$. 
As a left $\sch_F(n,d)$-module,
\begin{equation}
  \label{eq:wedge}
  \sch^-_F(n,d)=\bigoplus_{\lambda\in \Lambda(n,d)}\wedge^\lambda F^n.
\end{equation}

For each partition $\lambda$ of $d$ with at most $n$ parts, let $\Delta(\lambda)$ denote the $\sch_F(n,d)$-module known as the \emph{Weyl module} with highest weight $\lambda$ as in \cite[Section~1]{Donkin}.
A \emph{tilting module} is an $\sch_F(n,d)$-module $V$ such that both $V$ and its dual $V^*$ have filtrations by the Weyl modules $\Delta(\lambda)$.
Ringel \cite{MR1128706} showed that, for every such $\lambda$, there exists an indecomposable tilting module $M(\lambda)$ with unique highest weight $\lambda$.
A \emph{full tilting module} is a tilting module that contains $M(\lambda)$ as a direct summand for every partition $\lambda$ of $d$ with at most $n$ parts \cite[Section~3]{Donkin}. 
By Donkin \cite[Lemma~3.4]{Donkin}, \eqref{eq:wedge} implies:
\begin{lemma}
  \label{lm:tilting}
  The left module $\sch^-_F(n,d)$ is a full tilting module of $\sch_F(n,d)$.
\end{lemma}
\subsection{Twisted permutation representations}
\label{sec:twist-perm-repr}
Let $X$ be a finite set on which a group $G$ acts on the right (henceforth called a $G$-set).
The space $F[X]$ of $F$-valued functions on $X$ may be regarded as a representation of $G$:
\begin{equation}
  \label{eq:perm-rep}
  \rho_X(g)f(x) = f(x\cdot g), \text{ for } x\in X, \; g\in G, \text{ and } f\in F[X].
\end{equation}
Let $\chi$ be a multiplicative character $G\to F^*$.
One may twist the representation \eqref{eq:perm-rep} by $\chi$:
\begin{equation}
  \label{eq:twisted-action}
  \rho_X^\chi(g) f(x) = \chi(g)f(x\cdot g).
\end{equation}
Denote the representation space of this twisted action as $F[X]\otimes\chi$.

Suppose that $X$ and $Y$ are finite $G$-sets.
Given a function $\kappa:X\times Y\to F$, the \emph{integral operator} $\xi_\kappa:F[Y]\to F[X]$ associated to $\kappa$ is defined as
\begin{equation}
  \label{eq:integral-operator}
  \xi_\kappa f(x) = \sum_{y\in Y} \kappa(x,y)f(y), \text{ for } f\in F[Y].
\end{equation}
The function $\kappa$ is known as the \emph{integral kernel} of  $\xi_\kappa$.

If $Z$ is another finite $G$-set, $\kappa':X\times Y\to F$ and $\kappa'':Y\times Z \to F$ are functions.
Then
\begin{displaymath}
  \xi_{\kappa'} \circ \xi_{\kappa''} = \xi_{\kappa'*\kappa''},
\end{displaymath}
where $\kappa'*\kappa'':X\times Z\to F$ is the \emph{convolution product}
\begin{equation}
  \label{eq:convolution-product}
  \kappa'*\kappa''(x,z) = \sum_{y\in Y} \kappa'(x,y)\kappa''(y,z).
\end{equation}
We have (see \cite[Section~4.2]{rtcv}):
\begin{theorem}
  \label{theorem:integral_ops}
  For any finite $G$-spaces $X$ and $Y$, and any multiplicative character $\chi:G\to F^*$,
  \begin{multline}
    \label{eq:twisted-integral}
    \Hom_G(F[Y], F[X]\otimes \chi) \\
    = \{\xi_\kappa\mid \kappa:X\times Y\to F \text{ such that } \kappa(x\cdot g, y\cdot g)=\chi(g)\kappa(x,y)\}.
  \end{multline}
\end{theorem}
The identity \eqref{eq:twisted-integral} implies that
\begin{displaymath}
  \dim \Hom_G(F[Y], F[X]\otimes \chi) \leq |(X\times Y)/G|,
\end{displaymath}
with equality holding if $\chi$ is the trivial character.
However, if $g\in G$, and $(x,y)\in X\times Y$ are such that $(x\cdot g, y\cdot g)=(x, y)$, then if $\xi_\kappa\in \Hom_G(F[Y],F[X]\otimes \chi)$,
\begin{displaymath}
  \kappa(x,y) = \kappa(x\cdot g, y\cdot g) = \chi(g)\kappa(x,y),
\end{displaymath}
so that either $\chi(g)=1$ or $\kappa$ vanishes on the $G$-orbit of $(x,y)$.

For each element $x\in X$, let $G_x=\{g\in G\mid g\cdot x = x\}$, the stabilizer of $x$ in $G$.
\begin{definition}
  [Transverse Pair]
  \label{definition:transversality}
  A pair $(x,y)\in X\times Y$ is said to be \emph{transverse} with respect to $\chi$ if $G_x\cap G_y\subset \ker\chi$.
  If $(x,y)$ is a transverse pair with respect to $\chi$, we write $x\pitchfork y$.
\end{definition}
If $(x,y)$ is a transverse pair, then
\begin{displaymath}
  \kappa(x\cdot g, y\cdot g):=\chi(g)\kappa(x,y)
\end{displaymath}
is a well-defined non-zero function on the $G$-orbit of $(x,y)$.
Let
\begin{displaymath}
  X\pitchfork Y = \{(x,y)\in X\times Y\mid x\pitchfork y\}.
\end{displaymath}
Then $X\pitchfork Y$ is stable under the diagonal action of $G$ on $X\times Y$.
We have (see \cite[Theorem~4.2.3]{rtcv}):
\begin{theorem}
  \label{theorem:twisted-intertwiner}
  Let $X$ and $Y$ be finite $G$-sets, and $\chi:G\to F^*$ be a multiplicative character.
  For each orbit $O\in (X\pitchfork Y)/G$, choose a base point $(x_O,y_O)\in O$.
  Define
  \begin{displaymath}
    \kappa_O(x,y) =
    \begin{cases}
      \chi(g) & \text{if } x=x_O\cdot g \text{ and } y=y_O\cdot g \text{ for some } g\in G,\\
      0 & \text{otherwise}.
    \end{cases}
  \end{displaymath}
  For simplicity, write  $\xi_O$ for $\xi_{\kappa_O}$.
  Then the set
  \begin{displaymath}
    \{\xi_O\mid O \in (X\pitchfork Y)/G\}
  \end{displaymath}
  is a basis for $\Hom_G(F[Y],F[X]\otimes\chi)$.
  Consequently,
  \begin{displaymath}
    \dim \Hom_G(F[Y],F[X]\otimes \chi) = |(X\pitchfork Y)/G)|.
  \end{displaymath}
\end{theorem}
In the special case where $\chi$ is the trivial character, we get:
\begin{corollary}
  \label{corollary:untwisted-intertwiner}
  Let $X$ and $Y$ be finite $G$-sets.
  For each orbit $O$ in $(X\times Y)/G$ define
  \begin{displaymath}
    \kappa_O(x,y) =
    \begin{cases}
      1 &\text{if } (x,y)\in O,\\
      0 &\text{otherwise}.
    \end{cases}
  \end{displaymath}
  Write $\xi_O=\xi_{\kappa_O}$.
  Then the set
  \begin{displaymath}
    \{\xi_O\mid O\in (X\times Y)/G\}
  \end{displaymath}
  is a basis for $\Hom_G(F[Y],F[X])$.
  Consequently,
  \begin{displaymath}
    \dim \Hom_G(F[Y],F[X]) = |(X\times Y)/G|.
  \end{displaymath}
\end{corollary}
Given a function $\kappa:X\times Y\to F$, define
\begin{displaymath}
  \kappa^*(y,x)=\kappa(x,y) \text{ for $x\in X$, $y\in Y$}. 
\end{displaymath}
The following is easy to see:
\begin{lemma}
\label{lm:transpose}
For any $G$-set $X$, the map $\xi_\kappa\mapsto \xi_{\kappa^*}$ is an anti-involution on the algebra $\End_G(F[X])$.
\end{lemma} 
\subsection{Structure constants of the Schur algebra}\label{sect:scschur}
We recall the combinatorial interpretation of structure constants of the Schur algebra from \cite{structure}.
Let $[n]=\{1,\dotsc,n\}$ and
\begin{displaymath}
  I(n,d) = \{\underline i := (i_1,\dotsc,i_d)\mid i_s\in [n]\}.
\end{displaymath}
An element $w\in\SG_d$ acts on $I(n,d)$ by permuting the coordinates:
\begin{displaymath}
  (i_1,\dotsc,i_d)\cdot w = (i_{w(1)},\dotsc, i_{w(d)}).
\end{displaymath}
For $\underline i =(i_1,\dotsc,i_d)\in I(n,d)$, define
\begin{displaymath}
  e_{\underline i} = e_{i_1}\otimes \dotsb \otimes e_{i_d},
\end{displaymath}
where $e_i$ is the $i$th coordinate vector in $F^n$.
The vector space $(F^n)^{\otimes d}$ has a basis
\begin{displaymath}
  \{e_{\underline i}\mid \underline i\in I(n,d)\}.
\end{displaymath}
and $w\in \SG_d$ acts on a basis vector $e_{\underline i}$ as follows:
\begin{displaymath}
  w\cdot e_{\underline i}=w\cdot (e_{i_{1}}\otimes\cdots\otimes e_{i_{d}})=
  e_{i_{w^{-1}(1)}}\otimes\cdots\otimes e_{i_{w^{-1}(d)}}.
\end{displaymath}

Let $F[I(n,d)]$ denote the space of all $F$-valued functions on $I(n,d)$.
Mapping $e_{\underline i}$ to the indicator function of $\underline i\in I(n,d)$ defines an isomorphism of $(F^n)^{\otimes d}$ onto $F[I(n,d)]$.
Thus $(F^n)^{\otimes d}$ can be regarded as a permutation representation of $\SG_d$.

Let $B(n,d)$ denote the set of all configurations of $d$ distinguishable balls, numbered $1,\dotsc,d$ in $n$ boxes, numbered $1,\dotsc,n$.
The symmetric group $\SG_d$ acts on such configurations by permuting the $d$ balls.
An element of $B(n,d)$ is a set partition
\begin{displaymath}
  \{1,\dotsc,d\} = S_1\coprod \dotsb \coprod S_n,
\end{displaymath}
where $S_i$ is the set of balls in the $i$th box.
\begin{lemma}
  Given $\underline i\in I(n,d)$, let $b(\underline i)$ denote the balls-in-boxes configuration in $B(n,d)$ where the $i$th box contains the balls $\{s\mid i_s=i\}$.
  Then $b:I(n,d)\to B(n,d)$ is an $\SG_d$-equivariant bijection of $I(n,d)$ onto $B(n,d)$.
\end{lemma}
\label{sec:basis-schur-algebra}
By Corollary~\ref{corollary:untwisted-intertwiner}, a basis for $\sch_F(n,d)$ is indexed by orbits for the diagonal action of $\SG_d$ on $B(n,d)\times B(n,d)$.
\begin{definition}
  [Labelled bipartite multigraph]
  \label{definition:labelling}
  Let $[n]=\{1,\dotsc, n\}$ (as before) and $[n']=\{1',\dotsc,n'\}$.
  A labelling of a bipartite multigraph with vertex set $[n']\coprod [n]$ and $d$ edges is a function $l:[d] \to [n']\times [n]$ such that, for each $(i',j)\in [n']\times [n]$, the cardinality of $l^{-1}(i',j)$ is the number of edges joining $i'$ and $j$.
  In other words, labels are assigned to edges without distinguishing between edges joining the same pair of vertices.
\end{definition}
Given a pair $S=(S_1,\dotsc,S_n)$ and $T=(T_1,\dotsc,T_n)$ in $B(n,d)$, define a labelled bipartite graph $\gamma_{S,T}$ with multiple edges on the vertex set $[n']\coprod [n]$ as follows:
\begin{quote}
  There are $|S_j\cap T_i|$ edges between $i'$ and $j$, labelled by the numbers of the balls in $S_j\cap T_i$.
\end{quote}
The bipartite multigraph is always drawn in two rows, with the vertices from $[n']$ in the upper row and vertices from $[n]$ in the lower row, numbered from left to right.
Since the vertices are always labelled in this manner, the vertex labels can be omitted in the drawing.
\begin{example}
  \label{example:std-lab-multi}
  When $S=(\{1\}, \{2\},\{3,4,5\})$ and $T=(\{1,2,3\},\emptyset,\{4, 5\})$, the associated labelled multigraph is:
  \begin{displaymath}
    \xymatrix{
      \bullet \ar@{-}[d]|1\ar@{-}[dr]|2\ar@{-}[drr]|3 & \bullet & \bullet\ar@{=}[d]|{45}\\
      \bullet & \bullet & \bullet
    }
  \end{displaymath}
\end{example}
Clearly, $(S, T)\mapsto \gamma_{S,T}$ is a bijection from $B(n,d)\times B(n,d)$ onto the set of labelled bipartite multigraphs with vertex set $[n']\coprod [n]$ and $d$ edges.
The symmetric group $\SG_d$ acts on $B(n,d)\times B(n,d)$ by permuting labels.
Therefore the $\SG_d$ orbits in $B(n,d)\times B(n,d)$ are obtained by simply forgetting the labels, leaving only the underlying bipartite multigraph.
We write $\Gamma_{S,T}$ for the bipartite multigraph underlying $\gamma_{S,T}$.
Such a graph can also be represented by its \emph{adjacency matrix} (whose $(i,j)$th entry is the number of edges joining $i'$ and $j$), which is a matrix of non-negative integers that sum to $d$.

In view of Corollary~\ref{corollary:untwisted-intertwiner}, we recover a result of \cite{structure,MR1897636}:
\begin{theorem}
  \label{thm:basis_S}
  Let $M(n,d)$ denote the set of all bipartite multigraphs with vertex set $[n']\coprod [n]$ and $d$ edges.
  For each $\Gamma\in M(n,d)$, define $\xi_\Gamma\in \sch_F(n,d)=\End_{\SG_d}(F[B(n,d)])$ by
  \begin{displaymath}
    \xi_\Gamma f(S) = \sum_{\{T\mid \Gamma_{S,T}=\Gamma\}} f(T).
  \end{displaymath}
  Then
  \begin{displaymath}
    \{\xi_\Gamma\mid \Gamma\in M(n,d)\}
  \end{displaymath}
  is a basis for $\sch_F(n,d)$.
\end{theorem}
\begin{remark}
  If $(\underline i, \underline j)$ has image $\Gamma$ under the composition $I(n,d)^2\to B(n,d)^2\to M(n,d)$, then the basis element $\xi_{\underline i, \underline j}$ of \cite[Section~2.6]{MR2349209} coincides with the basis element $\xi_\Gamma$ of Theorem~\ref{thm:basis_S}.
\end{remark}

The structure constants $c^\Gamma_{\Gamma_1\Gamma_2}$ are defined by
\begin{displaymath}
  \xi_{\Gamma_1}\xi_{\Gamma_2}=\sum_{\Gamma\in M(n,d)}c^{\Gamma}_{\Gamma_1\Gamma_2} \xi_\Gamma.
\end{displaymath}
\begin{definition}
  \label{def:compatible}
  Let $l$, $l_1$ and $l_2$ be labellings of graphs $\Gamma$, $\Gamma_1$ and $\Gamma_2$ in $M(n,d)$, respectively.
  We say that $(l_1,l_2)$ is compatible with $l$ if, for all $s=1,\dotsc,d$, if we write $l_1(s)=(i_1',j_1)$ and $l_2(s)=(i_2',j_2)$, then
  \begin{enumerate}
  \item $i_2'=j_1$, and
  \item $l(s)=(i_1',j_2)$.
  \end{enumerate}
\end{definition}
We obtain yet another enumerative description of the structure constants of the Schur algebra (see also \cite[2.3(b)]{MR2349209} and \cite{structure,MR1897636}).
\begin{theorem}
  \label{theorem:schur-structure-const}
  Let $l$ be any labelling of $\Gamma$.
  The structure constant $c^\Gamma_{\Gamma'\Gamma''}$ is the number of pairs $(l_1,l_2)$ of labellings of $\Gamma_1$ and $\Gamma_2$ that are compatible with $l$.
\end{theorem}
Before giving a proof, we illustrate the theorem with a few examples.
\begin{example}
  \label{example:permuations}
  Let $w\in \SG_d$ be a permutation, and assume that $n\geq d$.
  Let $\Gamma(w)$ denote the bipartite graph where $(i',j)$ is an edge if and only if $1\leq i\leq d$ and $w(i)=j$.
  Then, for all $w_1,w_2\in \SG_d$, $\xi_{\Gamma(w_1)}\xi_{\Gamma(w_2)}=\xi_{\Gamma(w_1w_2)}$.
\end{example}
\begin{example}
  \label{example:stru-const}
  Consider
  \begin{equation*}
    \Gamma_1 = \vcenter{
      \xymatrix{
        \bullet \ar@{=}[d] & \bullet \ar@{-}[dl] \ar@{-}[d]\\
        \bullet & \bullet
      }
    },\quad
    \Gamma_2 = \vcenter{
      \xymatrix{
        \bullet \ar@{=}[d] \ar@{-}[dr] & \bullet \ar@{-}[d]\\
        \bullet & \bullet
      }
    }
  \end{equation*}
  To find $c^{\Gamma}_{\Gamma_1\Gamma_2}$, with
  \begin{equation*}
    \Gamma = \Gamma_3 = \vcenter{
      \xymatrix{
        \bullet \ar@3{-}[d] & \bullet \ar@{-}[d]\\
        \bullet & \bullet
      }
    },
  \end{equation*}
  choose any labelling of $\Gamma$, such as
  \begin{displaymath}
    \xymatrix{
      \bullet \ar@3{-}[d]|{123} & \bullet \ar@{-}[d]|{4}\\
      \bullet & \bullet
    }
  \end{displaymath}
  For this there are clearly three pairs of compatible labellings of $\Gamma_1$ and $\Gamma_2$, namely, we can choose which of the first three balls ends up in the second box of the middle row:
  \begin{equation}
    \label{eq:middle-box-3}
    \vcenter{\xymatrix{
        \bullet \ar@{=}[d]|{12} \ar@{-}[dr]|3 & \bullet \ar@{-}[d]|4\\
        \bullet \ar@{=}[d]|{12} & \bullet\ar@{-}[dl]|3 \ar@{-}[d]|4\\
        \bullet & \bullet
      }},\quad
    \vcenter{\xymatrix{
        \bullet \ar@{=}[d]|{13} \ar@{-}[dr]|2 & \bullet \ar@{-}[d]|4\\
        \bullet \ar@{=}[d]|{13} & \bullet\ar@{-}[dl]|2 \ar@{-}[d]|4\\
        \bullet & \bullet
      }}, \text{ and }
    \vcenter{\xymatrix{
        \bullet \ar@{=}[d]|{23} \ar@{-}[dr]|1 & \bullet \ar@{-}[d]|4\\
        \bullet \ar@{=}[d]|{23} & \bullet\ar@{-}[dl]|1 \ar@{-}[d]|4\\
        \bullet & \bullet
      }}.
  \end{equation}
  On the other hand, if
  \begin{equation*}
    \Gamma = \Gamma_4,
    \vcenter{
      \xymatrix{
        \bullet \ar@{=}[d] \ar@{-}[dr] & \bullet \ar@{-}[dl]\\
        \bullet & \bullet
      }
    },
  \end{equation*}
  we may take the labelling:
  \begin{equation}
    \label{eq:middle-box-4}
    \vcenter{
      \xymatrix{
        \bullet \ar@{=}[d]|{12} \ar@{-}[dr]|-<<{3} & \bullet \ar@{-}[dl]|-<<{4}\\
        \bullet & \bullet
      }
    },
  \end{equation}
  for which the only compatible labellings of $\Gamma_1$ and $\Gamma_2$ are:
  \begin{equation*}
    \xymatrix{
      \bullet \ar@{=}[d]|{12} \ar@{-}[dr]|3 & \bullet \ar@{-}[d]|4\\
      \bullet \ar@{=}[d]|{12} & \bullet\ar@{-}[dl]|4 \ar@{-}[d]|3\\
      \bullet & \bullet.
    }
  \end{equation*}
  It turns out that for no other $\Gamma \in \Gamma(n,d)$ is it possible to find even one compatible way of labelling $\Gamma_1$ and $\Gamma_2$, so we have:
  \begin{displaymath}
    \xi_{\Gamma_1}\xi_{\Gamma_2} = 3\xi_{\Gamma_3} + \xi_{\Gamma_4}.
  \end{displaymath}
\end{example}
\begin{example}
  \label{example:latin-squares}
  Let $F_{n,n}$ denote the complete bipartite graph with vertex set $[n']\coprod [n]$, where every vertex in $[n']$ is connected to every vertex in $[n]$.
  Then the coefficient of $\xi_{F_{n,n}}$ in $\xi_{F_{n,n}}\xi_{F_{n,n}}$ is the number of Latin squares of order $n$ \cite[Sequence~A002860]{oeis}.

  To see this, let $l$ be any labelling of the edges of $F_{n,n}$.
  Given labellings $l_1$ and $l_2$ be of $F_{n,n}$ that are compatible with $l$,
  define the $(i,j)$th entry of the Latin square associated to $(l_1,l_2)$ to be $k$ if $l^{-1}(i',j)=l_2^{-1}(k',j)=l_1^{-1}(i',k)$.
  Remarkably, the number of Latin squares of order $n$ is known only for $n=1,\dotsc,11$.
\end{example}
\begin{proof}
  [Proof of Theorem~\ref{theorem:schur-structure-const}]
  Given a labelling $l$ of $\Gamma$, define:
  \begin{displaymath}
    S_j = \cup_{i=1}^n l^{-1}(i',j), \text{ and } U_i = \cup_{j=1}^n l^{-1}(i',j).
  \end{displaymath}
  Then $S = (S_1,\dotsc,S_n)$, and $U=(U_1,\dotsc,U_n)$ are elements of $B(n,d)$, and
  by construction $\Gamma_{S,U}=\Gamma$.
  Now Equation~\eqref{eq:convolution-product} implies that
  \begin{equation}
    \label{eq:stru-const}
    c^\Gamma_{\Gamma_1\Gamma_2} = \#\{T\in B(n,d)\mid \Gamma_{S,T}=\Gamma_1\text{ and } \Gamma_{T,U}=\Gamma_2\}.
  \end{equation}
  Given $T\in B(n,d)$ contributing to the above count, define labellings $l_1$ and $l_2$ of $\Gamma_1$ and $\Gamma_2$ by:
  \begin{displaymath}
    l_1^{-1}(i',j) = S_j\cap T_i, \text{ and } l_2^{-1}(i',j)= T_j\cap U_i.
  \end{displaymath}
  Then $(l_1,l_2)$ is compatible with $l$.
  Conversely, for every pair $(l_1,l_2)$ compatible with $l$, take $T=(T_1,\dotsc,T_n)$ where
  \begin{displaymath}
    T_k = \cup_{i'=1}^n l_1^{-1}(i',k) = \cup_{j=1}^n l_2^{-1}(k',j).
  \end{displaymath}
  Then $T$ contributes to the count in \eqref{eq:stru-const}.
\end{proof}
\begin{example}
  In Example~\ref{example:stru-const}, the three compatible pairs of labels in \eqref{eq:middle-box-3} correspond to taking $T$ as $(\{1,2\},\{3,4\})$,
  $(\{1,3\},\{2,4\})$, and $(\{2,3\},\{1,4\})$, respectively, and the compatible pair of labels in \eqref{eq:middle-box-4} corresponds to $T=(\{1,2\},\{3,4\})$.
\end{example}
\subsection{A basis for $\sch_F^-(n,d)$}
\label{sec:basis-for-minus}
By Theorem~\ref{theorem:twisted-intertwiner}, a basis of $\sch^-_F(n,d)$ is indexed by orbits in $B(n,d)\pitchfork B(n,d)/\SG_d$.
Here $\pitchfork$ denotes transversality with respect to the sign character $\sgn:S_n\to \{\pm 1\}$ (see Definition~\ref{definition:transversality}).
\begin{lemma}
  A pair $(S,T)\in B(n,d)^2$ lies in $B(n,d)\pitchfork B(n,d)$ if and only if $\gamma_{S,T}$ is a simple bipartite graph.
\end{lemma}
\begin{proof}
  Let $S=(S_1,\dotsc,S_n)$, $T=(T_1,\dots,T_n)$.
  If $\gamma_{S,T}$ is not simple, then there exist indices $i$ and $j$ such that $S_j\cap T_i$ contains at least two elements, say $k$ and $l$.
  The transposition $(kl)\in \SG_d$ stabilizes $(S,T)$ but has $\sgn((kl))=-1$, so $(S,T)\notin B(n,d)\pitchfork B(n,d)$.

  However, if $\gamma_{S,T}$ is simple, then the simultaneous stabilizer of $S$ and $T$ in $\SG_d$ is trivial, so $(S,T)\in B(n,d)\pitchfork B(n,d)$.
\end{proof}
In order to specify a basis for $\sch^-_F(n,d)$ using Theorem~\ref{theorem:twisted-intertwiner}, we need to choose a base point for each $\SG_d$-orbit in $B(n,d)\pitchfork B(n,d)$.
We do this using the following definition:
\begin{definition}
  [Standard labelling of a bipartite simple graph]
  \label{definition:std-lab}
  Given a bipartite simple graph $\Gamma$ with vertex set $[n']\coprod[n]$, label each edge by its index when the edges $(i',j)$ are arranged in increasing lexicographic order, with priority given to the upper index, i.e., $(i',j)< (r',s)$ if either $i'< r'$ or $i'=r'$ and 
  $j<s$.
\end{definition}
\begin{example}
  \label{example:standard-labelling}
  Take
  \begin{displaymath}
    \Gamma =
    \vcenter{
      \xymatrix{
        \bullet\ar@{-}[dr] & \bullet\ar@{-}[dl] \ar@{-}[d] & \bullet \ar@{-}[dl]\\
        \bullet & \bullet & \bullet
      }
    }
  \end{displaymath}
  The edges, written in lexicographic order, are:
  \begin{displaymath}
    (1',2), (2',1), (2',2), (3',2).
  \end{displaymath}
  Therefore the standard labelling is:
  \begin{displaymath}
    \xymatrix{
      \bullet\ar@{-}[dr]|-<<1 & \bullet\ar@{-}[dl]|-<<2 \ar@{-}[d]|-3 & \bullet \ar@{-}[dl]|-4\\
      \bullet & \bullet & \bullet
    }
  \end{displaymath}
\end{example}
\begin{definition}
  [Sign of a labelling of a bipartite simple graph]
  Let $l_0$ denote the standard labelling of a simple bipartite graph $\Gamma$ on $[n']\coprod [n]$.
  Let $l:[d]\to [n']\times [n]$ be a labelling of $\Gamma$ (see Definition~\ref{definition:labelling}).
  The sign $\epsilon(\Gamma,l)$ of $l$ is the sign of the permutation on $[d]$ which takes $l_0(i)$ to $l(i)$ for each $i$.
\end{definition}
\begin{example}
  For the graph from Example~\ref{example:standard-labelling}, the labellings:
  \begin{displaymath}
    l_1 = \vcenter{\xymatrix{
      \bullet\ar@{-}[dr]|-<<4 & \bullet\ar@{-}[dl]|-<<2 \ar@{-}[d]|-3 & \bullet \ar@{-}[dl]|-1\\
      \bullet & \bullet & \bullet
    }} \text{ and }
    l_2 = \vcenter{\xymatrix{
      \bullet\ar@{-}[dr]|-<<1 & \bullet\ar@{-}[dl]|-<<3 \ar@{-}[d]|-4 & \bullet \ar@{-}[dl]|-2\\
      \bullet & \bullet & \bullet
    }}
  \end{displaymath}
  give rise to permutations $4231$ and $1342$ respectively, so that $\epsilon(\Gamma,l_1)=-1$, and $\epsilon(\Gamma,l_2) = +1$.
\end{example}
Recall, from Section~\ref{sec:basis-schur-algebra}, that $\gamma_{S,T}$ is a labelled bipartite graph associated to $(S, T)\in B(n,d)\times B(n,d)$, whose underlying unlabelled graph is denoted by $\Gamma_{S,T}$.
Let $l_{S,T}:[d]\to [n]\times [n']$ denote the labelling of $\gamma_{S,T}$, and write $\epsilon(\gamma_{S,T})$ for $\epsilon(\Gamma_{S,T},l_{S,T})$.
\begin{theorem}
  \label{thm:basis_S-}
  Let $N(n,d)$ denote the set of all bipartite simple graphs with vertex set $[n']\coprod [n]$ and $d$ edges.
  For each $\Gamma\in N(n,d)$, define $\zeta_\Gamma\in \sch_F^-(n,d) = \Hom_{\SG_d}(F[B(n,d)], F[B(n,d)]\otimes \sgn)$ by
  \begin{displaymath}
    \zeta_\Gamma f(S) = \sum_{\{T\mid \Gamma_{S,T} = \Gamma\}} \epsilon(\gamma_{S,T})f(T).
  \end{displaymath}
  The set
  \begin{displaymath}
    \{\zeta_{\Gamma}\mid \Gamma\in N(n,d)\}
  \end{displaymath}
  forms a basis of $\sch_F^-(n,d)$.
\end{theorem}
\begin{proof}
 Recall that we choose the pair $(S_{0},T_{0})$ corresponding to the standard labelling $l_{0}$ of 
 $\Gamma$ as the base point of the orbit associated to $\Gamma$. A pair $(S,T)$ is in the orbit of 
 $(S_{0},T_{0})$ if and only if $\Gamma_{S,T}=\Gamma$. And the sign of the permutation $w\in \SG_d$ such that 
 $S=S_{0}.w$ and $T=T_{0}.w$ is the sign of the labelled bipartite graph $\gamma_{S,T}$.
 So the integral kernel $\kappa_{\Gamma}$ of the operator $\zeta_{\Gamma}$ is:
 \begin{displaymath}
   \kappa_\Gamma(S,T) =
   \begin{cases}
     \epsilon(\gamma_{S,T}) & \text{if } \Gamma_{S,T}=\Gamma,\\
     0 & \text{otherwise}.
   \end{cases}
 \end{displaymath}
 So the theorem follows from Theorem~\ref{theorem:twisted-intertwiner}.
\end{proof}

Theorem~\ref{theorem:schur-structure-const} tells us how to multiply two elements of the subalgebra $\sch_F(n,d)$ of $\asch_F(n,d)$.
The remaining structure constants are given by the following theorem.
\begin{theorem}
  \label{theorem:stru-consts-alt}The remaining structure constants are given as follows:
  \begin{enumerate}
  \item \label{item:xi-zeta} Given $\Gamma_1\in M(n, d)$, and $\Gamma_2\in N(n,d)$,
    \begin{displaymath}
      \xi_{\Gamma_1}\zeta_{\Gamma_2} = \sum_{\Gamma\in N(n,d)} c^{\Gamma}_{\Gamma_1\Gamma_2} \zeta_\Gamma,
    \end{displaymath}
    where
    \begin{displaymath}
      c^{\Gamma}_{\Gamma_1\Gamma_2} = \sum_{l_1,l_2} \epsilon(\Gamma_2, l_2),
    \end{displaymath}
    and the sum runs over all labellings $l_1$ and $l_2$ of $\Gamma_1$ and $\Gamma_2$ respectively, that are compatible with the standard labelling $l$ of $\Gamma$.
  \item Given $\Gamma_1\in N(n, d)$, and $\Gamma_2\in M(n,d)$,
    \begin{displaymath}
      \zeta_{\Gamma_1}\xi_{\Gamma_2} = \sum_{\Gamma\in N(n,d)} c^{\Gamma}_{\Gamma_1\Gamma_2} \zeta_\Gamma,
    \end{displaymath}
    where
    \begin{displaymath}
      c^{\Gamma}_{\Gamma_1\Gamma_2} = \sum_{l_1,l_2} \epsilon(\Gamma_1, l_1),
    \end{displaymath}
    and the sum runs over all labellings $l_1$ and $l_2$ of $\Gamma_1$ and $\Gamma_2$ respectively, that are compatible with the standard labelling $l$ of $\Gamma$.
  \item Given $\Gamma_1\in N(n, d)$, and $\Gamma_2\in N(n,d)$,
    \begin{displaymath}
      \zeta_{\Gamma_1}\zeta_{\Gamma_2} = \sum_{\Gamma\in M(n,d)} c^{\Gamma}_{\Gamma_1\Gamma_2} \xi_\Gamma,
    \end{displaymath}
    where
    \begin{displaymath}
      c^{\Gamma}_{\Gamma_1\Gamma_2} = \sum_{l_1,l_2} \epsilon(\Gamma_1, l_1)\epsilon(\Gamma_2, l_2),
    \end{displaymath}
    and the sum runs over all labellings $l_1$ and $l_2$ of $\Gamma_1$ and $\Gamma_2$ respectively, that are compatible with a fixed labelling $l$ of $\Gamma$.
  \end{enumerate}
\end{theorem}
\begin{proof}
  Given a labelling $l$ of $\Gamma$, construct $S$ and $U$ in $B(n,d)$ as in the proof of Theorem~\ref{theorem:schur-structure-const}.
  Define $\kappa_{\Gamma_2}:B(n,d)\times B(n,d)\to F$ by
  \begin{displaymath}
    \kappa_{\Gamma_2}(T, U) = \epsilon(\gamma_{T,U}).
  \end{displaymath}
  Then $\zeta_{\Gamma_2}$ is the integral operator $\xi_{\kappa_{\Gamma_2}}$, as in \eqref{eq:integral-operator}.
  Then, by Equation~\eqref{eq:convolution-product}, the structure constant in Part~\eqref{item:xi-zeta} of the theorem is given by:
  \begin{displaymath}
    c^\Gamma_{\Gamma_1\Gamma_2} = \sum_T \zeta_{\Gamma_2}(T, U),
  \end{displaymath}
  where the sum runs over all $T\in B(n,d)$ such that $\Gamma_{S,T}=\Gamma_1$, and $\Gamma_{T,U}=\Gamma_2$.
  Defining labelling $l_1$ and $l_2$ of $\Gamma_1$ and $\Gamma_2$ as in the proof of Theorem~\ref{theorem:schur-structure-const}, we find that $\zeta_{\Gamma_2}(T,U)=\epsilon(\Gamma_2,l_2)$, proving \eqref{item:xi-zeta}.
  The proofs of the remaining assertions are similar.
\end{proof}
\begin{definition}
  Given $\Gamma\in M(n,d)\sqcup N(n,d)$, we define $\Gamma^*$ to be the horizontal reflection of 
  $\Gamma$, i.e., $i'$ is connected to $j$ in $\Gamma^*$ if and only if $j'$ is connected to $i$ in $\Gamma$. The operation $*$ on the set $M(n,d)\sqcup N(n,d)$ is an involution.
\end{definition}
\begin{lemma}
  \label{lemma:anti-invol}
  For every $\Gamma\in N(n,d)$, let $l_0$ denote its standard labelling.
  Let $l_0^*$ denote the labelling of $\Gamma^*$ given by $l_0^*(i',j)=l_0(j',i)$.
  Then the linear map $\asch_F(n,d)\to \asch_F(n,d)$ defined by:
  \begin{align*}
    \xi_\Gamma & \mapsto \xi_{\Gamma^*} \text{ for $\Gamma\in M(n,d)$},\\
    \zeta_\Gamma & \mapsto \epsilon(\Gamma^*,l_0^*)\zeta_{\Gamma^*} \text{ for $\Gamma\in N(n,d)$}
  \end{align*}
  is an anti-involution of $\asch_F(n,d)$.
\end{lemma}
\begin{remark}
  The above involution, when restricted to the Schur algebra, is the same as the one described by Green \cite[Section~2.7]{MR2349209}.
\end{remark}
\begin{proof}
  We show that the linear map in Lemma~\ref{lemma:anti-invol} is the same as the anti-involution in Lemma~\ref{lm:transpose} with $X=B(n,d)$ and $G=\AG_d$.

  For $\Gamma\in M(n,d)$, $\xi_{\Gamma}$ is the integral operator with kernel:
  \begin{displaymath}
    \kappa_{\Gamma}(S,T)= \begin{cases}
      1 & \text{ if } \Gamma_{S,T}=\Gamma,\\
      0 & \text{ otherwise}.
    \end{cases}
  \end{displaymath}
  Since $\Gamma_{T,S}=\Gamma_{S,T}^*$,  $\kappa_\Gamma^*=\kappa_{\Gamma^*}$.

  For $\Gamma\in N(n,d)$, $\zeta_\Gamma$ is the integral operator with kernel:
  \begin{displaymath}
    \kappa_\Gamma(S,T) =
    \begin{cases}
      \epsilon(\gamma_{S,T}) & \text{ if } \Gamma_{S,T}=\Gamma,\\
      0 & \text{ otherwise}.
    \end{cases}
  \end{displaymath}
  Thus, if $\gamma_{S,T}=(\Gamma,l_0)$, then $\gamma_{T,S}=(\Gamma^*,l_0^*)$.
  Therefore,
  \begin{align*}
    \kappa_{\Gamma^*}(T,S)& = \epsilon(\gamma_{T,S})\\
    & = \epsilon(\Gamma^*,l_0^*)\kappa_\Gamma(S,T).
  \end{align*}
  So the kernels $\kappa_{\Gamma^*}$ and $\epsilon(\Gamma^*,l_0^*)\kappa_\Gamma^*$ coincide at $(T,S)$, and hence on its entire $\SG_d$-orbit in $B(n,d)$.
\end{proof}

We illustrate the above results with an example that will be used in the proof of Lemma~\ref{lemma:zetazeta}.
\begin{example}
  \label{example:zetazeta}
  Recall that $\Lambda(n,d)$ denotes the set of all weak compositions of $d$ with at most $n$ parts.
  For $\lambda\in \Lambda(n,d)$ with $n\geq d$, let $\Gamma_\lambda\in N(n,d)$ denote the bipartite graph where $i'$ is connected to $j$ if
  \begin{displaymath}
    \lambda_1 + \dotsb +\lambda_{i'-1}< j \leq \lambda_1 + \dotsb + \lambda_{i'}.
  \end{displaymath}

  Then we have:
  \begin{displaymath}
    \zeta_{\Gamma_\lambda} \zeta_{\Gamma_\lambda^*} = \lambda_1!\dotsb \lambda_n! \xi_{\Gamma^0_\lambda},
  \end{displaymath}
  where $\Gamma^0_\lambda\in M(n,d)$ is the bipartite multigraph where $i'$ is connected to $i$ by $\lambda_i$ edges.
  For example,
  \begin{displaymath}
    \Gamma_{(2,1)}^* =
    \vcenter{\xymatrix{
        \bullet \ar@{-}[d]\ar@{-}[dr] & \bullet \ar@{-}[dr] & \bullet\\
        \bullet & \bullet & \bullet
      }},
    \text{ and }
    \Gamma_{(2,1)} =
    \vcenter{\xymatrix{
        \bullet \ar@{-}[d] & \bullet \ar@{-}[dl] & \bullet \ar@{-}[dl]\\
        \bullet & \bullet & \bullet
      }}
  \end{displaymath}
  For a labelling $l_1 = \vcenter{\xymatrix{
      \bullet \ar@{-}[d]|a\ar@{-}[dr]|b & \bullet \ar@{-}[dr]|c & \bullet\\
      \bullet & \bullet & \bullet
    }}$
  of $\Gamma^*_{(2,1)}$ only the labelling\linebreak $l_2 = \vcenter{\xymatrix{
      \bullet \ar@{-}[d]|a & \bullet \ar@{-}[dl]|b & \bullet \ar@{-}[dl]|c\\
      \bullet & \bullet & \bullet
    }}$
  of $\Gamma_{(2,1)}$, and the labelling $l = \vcenter{\xymatrix{
      \bullet \ar@{=}[d]|{ab} & \bullet \ar@{-}[d]|c & \bullet\\
      \bullet & \bullet & \bullet
    }}$
  of $\Gamma^0_{(2,1)}$ are such that $(l_1,l_2)$ are compatible with $l$.
  Moreover, $\epsilon(\Gamma_{(2,1)}^*,l_1)=\epsilon(\Gamma_{(2,1)},l_2)$.
  Interchanging the labels $a$ and $b$ in $l_{1}$ and $l_{2}$, respectively, gives another pair of labels compatible with $l$, so that $\zeta_{\Gamma_{\lambda}}\zeta_{\Gamma_{\lambda}^*} = 2\xi_{\Gamma_\lambda^0}$.
\end{example}
The remaining results in this section help us understand the structure of $\sch^-_F(n,d)$ as an $\sch_F(n,d)$-module.
Some of them will play an important role in understanding Koszul duality (Section~\ref{sec:koszul-duality-schur}).

The notion of standard labelling (Definition~\ref{definition:std-lab}) of graphs in $N(n,d)$ can be extended to graphs in $M(n,d)$ as follows: when an edge $(i',j)$ occurs with multiplicity $m$, it is simply listed $m$ times when the edges are arranged in lexicographic order with priority given to the upper index.
Example~\ref{example:std-lab-multi} is the standard labelling of its underlying graph.
\begin{definition}
  For $n\geq d$, define the following simple bipartite graphs associated to $\Gamma\in M(n,d)$:
  \begin{enumerate}
  \item Let $D(\Gamma)\in N(n,d)$ be the graph with edges $(i',s)$ for every edge $(i',j)$ with label $s$ under the standard labelling of $\Gamma$.
  \item Let $U(\Gamma)\in N(n,d)$ be the graph with edges $(s',j)$ for every edge $(i',j)$ with label $s$ under the standard labelling of $\Gamma$.
\end{enumerate}
\end{definition}
\begin{example}
  Let $n=5$, $d=5$, and $\Gamma$ (with its standard labelling) is given by:
  \begin{displaymath}
    \vcenter{\xymatrix{
        \bullet\ar@{-}[d]|1\ar@{-}[dr]|2 & \bullet \ar@{-}[dr]|-<<{3} & \bullet \ar@{=}[dl]|-<<{45} & \bullet & \bullet\\
        \bullet & \bullet & \bullet & \bullet & \bullet
      }}
  \end{displaymath}
  Then
  \begin{displaymath}
    D(\Gamma)=\vcenter{\xymatrix{
        \bullet \ar@{-}[d]|1 \ar@{-}[dr]|2 & \bullet \ar@{-}[dr]|3 & \bullet \ar@{-}[dr]|4\ar@{-}[drr]|5 & \bullet & \bullet \\
        \bullet & \bullet & \bullet & \bullet & \bullet
      }}
  \end{displaymath}
  \text{and}
  \begin{displaymath}
    U(\Gamma)=\vcenter{\xymatrix{
        \bullet \ar@{-}[d]|1 & \bullet \ar@{-}[d]|2 & \bullet \ar@{-}[d]|-<<3 & \bullet \ar@{-}[dll]|-<<<<4 & \bullet\ar@{-}[dlll]|5\\
        \bullet & \bullet & \bullet & \bullet & \bullet
      }}
  \end{displaymath}
\end{example}
The significance of the elements $U(\Gamma)$ and $D(\Gamma)$, for $\Gamma\in M(n,d)$, is elaborated in the 
following lemmas.
\begin{lemma}\label{lm:bmodule}
  Let $n\geq d$ and $\Gamma\in N(n,d)$. Then $\zeta_{\Gamma}=\xi_{U(\Gamma)}\zeta_{\Gamma_{{\lambda}_{0}}}\xi_{D(\Gamma)}$, where $\lambda_{0}=(1^d,0^{n-d})\in\Lambda(n,d)$. Consequently, $\sch_F^-(n,d)$ is a cyclic $(\sch_F(n,d),\sch_F(n,d))$-bimodule.
\end{lemma}
\begin{proof}
  This can be done in two steps.
  Firstly, $\zeta_{\Gamma_{{\lambda}_{0}}}\xi_{D(\Gamma)}=\zeta_{D(\Gamma)}$, and secondly $\xi_{U(\Gamma)}\zeta_{D(\Gamma)}=\zeta_\Gamma$.
  We indicate the proof of the second identity (the first is similar):
  Let $l_0$, $l_1$, and $l_2$ be the standard labellings of $\Gamma$, $U(\Gamma)$ and $D(\Gamma)$ respectively. 
  The labellings $(l_{1},l_{2})$ of $D(\Gamma)$ and $U(\Gamma)$ are the only ones that are compatible with $l_{0}$.
  This is because, for the edge $(i',j)$ of $\Gamma$ labelled $s$, $s$ is the unique vertex such $i'$ is connected to $s$ in $D(\Gamma)$ and $s'$ is connected to $j$ in $U(\Gamma)$.
  The identity now follows from \eqref{item:xi-zeta}.
\end{proof}
Similarly, we have:
\begin{lemma}\label{lm:lrmodule}
  For $n\geq d$ and $\Gamma\in N(n,d)$, we have
  $\zeta_{\Gamma}=\zeta_{U(\Gamma)}\xi_{D(\Gamma)}$ and $\zeta_{\Gamma}=\xi_{U(\Gamma)}\zeta_{D(\Gamma)}$.
\end{lemma}
\begin{corollary}
  As a left $\sch_F(n,d)$-module, $\sch^-_F(n,d)$ is generated by $\{\zeta_{\Gamma_\lambda^*}\mid \lambda\in \Lambda(n,d)\}$, and as a right $\sch_F(n,d)$-module, it is generated by $\{\zeta_{\Gamma_\lambda}\mid \lambda\in \Lambda(n,d)\}$.  Here $\Gamma_\lambda$ is the graph associated to $\lambda$ in Example~\ref{example:zetazeta}.
\end{corollary}
\begin{proof}
  For any $\Gamma\in M(n,d)$, $D(\Gamma)$ is of the form $\Gamma_\lambda^*$ for some $\lambda\in \Lambda(n,d)$, so the statement for left modules follows from the second identity in Lemma~\ref{lm:lrmodule}.
  The statement for right modules follows by applying Lemma~\ref{lemma:anti-invol} to the first identity in Lemma~\ref{lm:lrmodule}.
\end{proof}
\begin{lemma}\label{lm:inj}
  Let $n\geq d$ and $\Gamma\in M(n,d)$. Then, for $\Gamma'\in M(n,d)$, the structure constant of $\zeta_{U(\Gamma)}$ in the product $\xi_{\Gamma'}\zeta_{D(\Gamma)^*}$ is $\delta_{\Gamma',\Gamma}$.
\end{lemma}
\begin{proof}
  The edge $(i',j)$ with label $s$ in the standard labelling of $\Gamma$ gives rise to an edge $(s',j)$ with standard label $s$ in $U(\Gamma)$.
  The graph $D(\Gamma)^*$ has only one edge originating at $s'$, namely $(s',i)$.
  Therefore, for any compatible pair $(l_1,l_2)$ of labellings of $D(\Gamma)^*$ and $\Gamma'$, this edge must have label $s$.
  Thus $\Gamma'$ must have an edge $(i',j)$ labelled $s$.
  In other words, $\Gamma'=\Gamma$ and $l_2$ is its standard labelling.
\end{proof}
\section{Abstract Koszul duality}
\subsection{The algebra}
\label{sec:alg}
Recall \cite[Chapter III, Section~3.1]{MR1727844} that a $\ZZ/2\ZZ$ grading on a ring $\asch$ is a decomposition $\asch = \sch\oplus \sch^-$ into additive subgroups such that  $\sch$ is a subring, $\sch^-$ is closed under left and right multiplication by elements of $\sch$, and for any $\alpha,\beta\in \sch^-$, $\alpha\beta\in \sch$.
This $\ZZ/2\ZZ$-grading gives rise to:
\renewcommand{\theenumi}{\alph{enumi}}
\setcounter{enumi}{0}
\begin{enumerate}
\item \label{item:1} an $(\sch,\sch)$-bimodule structure on $\sch^-$,
\item \label{item:2} and an $(\sch,\sch)$-bimodule homomorphism $\phi:\sch^-\otimes_{\sch} \sch^-\to \sch$ (induced by the $\sch$-balanced bilinear map $(\alpha,\beta)\mapsto \alpha\beta$ for $\alpha, \beta\in \sch^-$).
\end{enumerate}
\renewcommand{\theenumi}{\thetheorem.\alph{enumi}}
\begin{example}
  \label{example:schur-graded}
  We may take $\asch=\asch_F(n,d)=\End_{\AG_d}((F^n)^{\otimes d})$, and $\
  \sch=\sch_F(n,d)=\End_{\SG_d}((F^n)^{\otimes d})$, for any field $F$ with characteristic different from $2$.
\end{example}
\subsection{Modules}
\label{sec:modules}
Let $M$ be an $\asch$-module.
The $\asch$-module structure can be viewed as a linear map:
\begin{displaymath}
  \asch\otimes_\ZZ M = (\sch\oplus \sch^-)\otimes_\ZZ M = (\sch\otimes_\ZZ M)\oplus (\sch^-\otimes_\ZZ M)\to M.
\end{displaymath}
So $M$ is an $\sch$-module, and restriction of the module action $\asch\otimes_\ZZ M\to M$ to $\sch^-\otimes_\ZZ M$ induces an $\sch$-module homomorphism:
\begin{equation}
  \label{eq:3}
  \theta_M:\sch^-\otimes_{\sch} M\to M.
\end{equation}
Furthermore, this homomorphism $\theta_M$ has the property that the diagram
\begin{equation}
  \label{eq:1}
  \xymatrix{
    \sch^-\otimes_{\sch} \sch^-\otimes_{\sch} M \ar[rr]^{\phi\otimes \id_M}\ar[d]_{\id_{\sch^-}\otimes \theta_M}\ar@{}[drr] && S\otimes_{\sch} M \ar@{=}[d]\\
    \sch^-\otimes_{\sch} M \ar[rr]_{\theta_M} && M
  }
\end{equation}
commutes.
\begin{definition}
  Given an $(\sch,\sch)$-bimodule $\sch^-$ and an $(\sch,\sch)$-bimodule homomorphism $\phi:\sch^-\otimes_{\sch} \sch^-\to \sch$,
  for an $\sch$-module $N$, an $\sch$-module homomorphism $\theta:\sch^-\otimes_{\sch} N\to N$ is said to be compatible with $\phi$ if the diagram
  \begin{equation}
    \label{eq:2}
    \xymatrix{
      \sch^-\otimes_{\sch} \sch^-\otimes_{\sch} N \ar[rr]^{\phi\otimes \id_N}\ar[d]_{\id_{\sch^-}\otimes \theta}\ar@{}[drr] && S\otimes_{\sch} N \ar@{=}[d]\\
      \sch^-\otimes_{\sch} N \ar[rr]_{\theta} && N
    }
  \end{equation}
  commutes.
\end{definition}
\subsection{Duality}
\label{sec:duality}
Let $\asch=\sch\oplus \sch^-$ be as before.
Define a functor $D:\mod \sch\to \mod \sch$ by:
\begin{displaymath}
  D(M) = \sch^-\otimes_{\sch} M,
\end{displaymath}
for every $\sch$-module $M$.
Given $\sch$-modules $M$ and $N$, and an $S$-module homomorphism $f:M\to N$, let $D(f) = \id_{\sch^-}\otimes f: D(M)\to D(N)$.
We call the  resulting functor $D:\mod{\sch}\to \mod{\sch}$ an \emph{abstract Koszul duality functor}.
In Section \ref{sec:koszul-duality-schur} it will be shown that, in the setting of Example \ref{example:schur-graded} (the alternating Schur algebra), abstract Koszul duality is essentially the Koszul duality functor of Krause\cite{MR3077659}.

The commutative diagram \eqref{eq:2}, defining the compatibility of $\theta$ with $\phi$, can be rewritten in terms of abstract Koszul duality as:
\begin{equation}
  \label{eq:compatibility}
  \xymatrix{
    D^2(N) \ar[rr]^{\phi\otimes \id_N}\ar[d]_{D(\theta)}\ar@{}[drr] && N\\
    D(N) \ar[urr]_{\theta} &&
  }
\end{equation}
\begin{definition}
  Given an $(\sch,\sch)$-bimodule $\sch^-$ and an $(\sch,\sch)$-bimodule homomorphism $\phi:\sch^-\otimes_{\sch} \sch^-\to \sch$, let $\mod{(\sch,\phi)}$ denote the category whose objects are pairs $(N,\theta)$, where $N$ is an $\sch$-module, and $\theta:D(N)\to N$ is compatible with $\phi$.
  A morphism $(N,\theta)\to(N',\theta')$ is an $\sch$-module homomorphism $f:N\to N'$ such that the diagram
  \begin{displaymath}
    \xymatrix{
      D(N) \ar[r]^\theta \ar[d]_{D(f)} & N\ar[d]^f\\
      D(N') \ar[r]^{\theta'} & N'
    }
  \end{displaymath}
  commutes.
\end{definition}
\begin{theorem}
  \label{thm:iso}
  Given an $\asch$-module $M$, let $\theta_M$ be as in \eqref{eq:3}.
  Then $M\mapsto (M,\theta_M)$ is an isomorphism of categories $\mod{\asch}\to \mod{(\sch,\phi)}$.
\end{theorem}
\begin{proof}
  Given an object $(N,\theta)$ in $\mod{(\sch,\phi)}$, the compatibility of $\theta$ with $\phi$ allows the $\sch$-module structure on $N$ to be extended to an $\asch$-module structure.
  This constructs the inverse of the functor in the theorem.
\end{proof}
Given an $\sch$-module $N$, the morphism $\phi:\sch^-\otimes_{\sch} \sch^-\to \sch$ gives rise to a natural transformation $\eta: D^2 \to \id_{\mod \sch}$, defined as the composition:
\begin{equation}
  \label{eq:4}
  \xymatrix{
    D^2(N) \ar[rr]^{\phi\otimes \id_N} \ar[drr]_{\eta_N} && \sch\otimes_{\sch} N \ar@{=}[d]\\
    & & N
  }
\end{equation}
\begin{theorem}
  \label{theorem:phi-iso_con}
  Let $\asch$, $\sch$, $\sch^-$ and $\phi$ be as in Section~\ref{sec:alg}.
  The following are equivalent:
  \begin{enumerate}
  \item\label{item:3} The map $\phi:\sch^-\otimes_{\sch} \sch^-\to \sch$ is an isomorphism.
  \item\label{item:4} The natural transformation $\eta: D^2 \to \id_{\mod \sch}$ is a natural isomorphism.
  \item\label{item:5} For every object $(N,\theta)$ in $\mod{(\sch,\phi)}$, $\theta:\sch^-\otimes_{\sch}N\to N$ is an isomorphism of $\sch$-modules.
  \end{enumerate}
\end{theorem}
\begin{proof}
  To see that \eqref{item:3} implies \eqref{item:4}, observe from the diagram \eqref{eq:4} that if $\phi$ is an isomorphism, then $\eta_N$ is an isomorphism for every $N$.
  It follows that $\eta$ is a natural isomorphism.
  For the converse, taking $N=\sch$, the commutativity of \eqref{eq:4} shows that $\phi$ is an isomorphism.

  To see that \eqref{item:3} implies \eqref{item:5}, note that the commutativity of \eqref{eq:2} implies that, if $\phi$ is an isomorphism, then $\theta$ is an epimorphism, and $\id_{\sch^-}\otimes \theta$ is a monomorphism.
  Since tensoring is a right-exact functor, it follows that $\id_{\sch^-}\otimes \theta$ is also an epimorphism, hence an isomorphism.
  Since $\phi\otimes \id_N$ is also an isomorphism the inverse of $\theta$ can be constructed by reversing the arrows in \eqref{eq:2}.
  For the converse, just take $N=\sch$ in \eqref{item:5}.
\end{proof}
\subsection{Abstract Ringel duality}
Let $\sch^-$ be an $(\sch,\sch)$-bimodule. Denote the left $\sch$-module $\sch^-$ by $_{\sch}\sch^-$.
For a left $\sch$-module $M$, the homomorphism space $\Hom_{\sch}(_{\sch}\sch^-,M)$ inherits the structure of a left $\sch$-module from the right $\sch$-module structure on $\sch^-$.
Motivated by \cite[Section~6]{MR1128706}, we call the functor
\begin{displaymath}
  \Hom_{\sch}(_{\sch}\sch^-,-):\mod{\sch}\to \mod{\sch}
\end{displaymath}
the abstract Ringel duality functor on $\mod{\sch}$.
It is clear that the abstract Koszul duality functor is the left adjoint of abstract Ringel duality functor.
\subsection{Abstract simple modules}
\label{sec:abs-simp}
In general, it is not clear how simple $\asch$-modules can be classified using simple $\sch$-modules and Koszul duality.
In this section, we give some results in this direction.
These are enough to give a complete solution in the semisimple case.

Let $M$ be a simple $\sch$-module. We consider the following cases:
\subsubsection{$DM$ is isomorphic to $M$}
\label{sec:m-simple-dm}
If $\eta_M:D^2M\to M$ is zero, then $(M,0)$ (where $0$ is the zero map from $DM\to M$) is the unique $\phi$-compatible morphism.
Otherwise, any non-zero morphism $\theta:DM\to M$ is an isomorphism.
Schur's lemma implies that $\theta\circ D\theta=a\eta_M$ for some $a\in (\End_{\sch} M)^*$ (the multiplicative group of non-zero elements in the division algebra $\End_{\sch} M$).
If $a$ has a square root in $(\End_{\sch} M)^*$, then $\theta$ can be normalized to make it a $\phi$-compatible morphism.
Moreover, after normalization, $\pm \theta$ are two $\phi$-compatible morphisms, leading to two non-isomorphic $\asch$-modules.
Also, in this case, $(M,\pm\theta)$ are simple, because their restrictions to $\sch$ are simple.
On the other hand, if $a$ does not have a square root in $\End_{\sch}(M)$, then there is no simple $\asch$-module whose restriction to $\sch$ is isomorphic to $M$.
\subsubsection{$DM=0$}
\label{sec:m-simple-dm=0}
In this case $(M,0)$ is the unique $\asch$-module whose restriction to $\sch$ is isomorphic to $M$.
\subsubsection{$DM$ is simple, but not isomorphic to $M$, $\eta_M\neq 0$}
\label{sec:m-simple-dm-1}
Let $\tilde M = DM\oplus M$.
We have $D\tilde M = D^2M\oplus DM$.
Any morphism $\theta:D\tilde M\to \tilde M$ can be written in matrix form as
\begin{displaymath}
  \theta =
  \begin{pmatrix}
    X & Y \\ Z & W
  \end{pmatrix},
\end{displaymath}
where $X:D^2M\to DM$, $Y: DM\to DM$, $Z:D^2M\to M$, and $W:DM\to M$.
By Schur's lemma, $W=0$.
The compatibility of $\theta$ with $\phi$ becomes:
\begin{displaymath}
  \begin{pmatrix}
    X & Y\\
    Z & 0
  \end{pmatrix}
  \begin{pmatrix}
    DX & DY\\
    DZ & 0
  \end{pmatrix}
  =
  \begin{pmatrix}
    \eta_{DM} & 0\\
    0 & \eta_M
  \end{pmatrix}.
\end{displaymath}
Multiplying out the left hand side gives:
\begin{displaymath}
  \begin{pmatrix}
    XDX + YDZ & XDY\\
    ZDX & ZDY
  \end{pmatrix}
  =
  \begin{pmatrix}
    \eta_{DM} & 0\\
    0 & \eta_M
  \end{pmatrix}.
\end{displaymath}
Since $\eta_M\neq 0$, $DY\neq 0$, and so $Y\neq 0$.
Since $DM$ is simple, by Schur's lemma, $Y$ is invertible.
Since $D$ is a functor, $DY$ is also invertible.
Hence, equality of top right entries implies that $X=0$.
Moreover, $Z=\eta_M DY^{-1}$.
In other words, $\theta$ is of the form:
\begin{displaymath}
  \theta_Y =
  \begin{pmatrix}
    0 & Y\\
    \eta_MDY^{-1} & 0
  \end{pmatrix}.
\end{displaymath}
\begin{lemma}
  \label{lemma:YYprime}
  For all $Y, Y'\in (\End_{\sch}M)^*$, $\Hom_{\asch}((\tilde M, \theta_Y), (\tilde M, \theta_{Y'}))$ is non-zero, and $\End_{\asch}(\tilde M, \theta_Y)$ is a division ring.
\end{lemma}
\begin{proof}
  Any $\asch$-module morphism $(\tilde M, \theta_y)\to (\tilde M, \theta_{y'})$ can be written in matrix form as
  \begin{displaymath}
    \begin{pmatrix}
      X & 0\\
      0 & W
    \end{pmatrix},
    \text{ where $X\in \End_{\sch}DM$ and $W\in \End_{\sch}M$},
  \end{displaymath}
  and must satisfy:
  \begin{displaymath}
    \begin{pmatrix}
      X & 0\\
      0 & W
    \end{pmatrix}
    \begin{pmatrix}
      0 & Y\\
      \eta_MDY^{-1} & 0
    \end{pmatrix}
    =
    \begin{pmatrix}
      0 & Y'\\
      \eta_MDY^{\prime-1} & 0
    \end{pmatrix}
    \begin{pmatrix}
      DX & 0\\
      0 & DW
    \end{pmatrix}
  \end{displaymath}
  We get $XY=Y'DW$, and $W\eta_MDY^{-1}=\eta_MDY^{\prime-1}DX$.
  Taking $W=\id_M$, and $X=Y'Y^{-1}$ gives a non-zero element of $\Hom_{\asch}((\tilde M, \theta_Y), (\tilde M, \theta_{Y'}))$.
  When $Y=Y'$, then we have $X=YDWY^{-1}$, so that $X$ is non-zero (and hence invertible) if and only if $W$ is.
  It follows that every non-zero element of $\End_{\asch}(\tilde M, \theta_Y)$ is invertible.
\end{proof}
\begin{lemma}
  \label{lemma:astensor}
  Let $M$ be an $\sch$-module.
  Then the $\asch$-module $\asch\otimes_{\sch} M$ is isomorphic to $(DM\oplus M, \theta)$, where $\theta$ is given by the matrix:
  \begin{displaymath}
    \theta =
    \begin{pmatrix}
      0 & \id_{DM}\\
      \eta_M & 0
    \end{pmatrix}.
  \end{displaymath}
\end{lemma}
\begin{proof}
  Note that
  \begin{displaymath}
    \asch\otimes_{\sch} M = (\sch^-\otimes_{\sch}M) \oplus (\sch\otimes_{\sch} M) = DM\oplus M.
  \end{displaymath}
  The map $\theta$ comes from the action of $\sch^-$ on this $\asch$-module, which gives $\eta_M:D^2M\to M$ on the first summand, and $\id_{DM}:DM\to DM$ on the second summand.
\end{proof}
\begin{theorem}
  \label{theorem:DMplusM}
  The $\asch$-module $(DM\oplus M, \theta_Y)$ defined above is isomorphic to $\asch\otimes_{\sch} M$ for every $Y\in (\End_{\sch}DM)^*$.
  Consequently, whenever $M$ and $DM$ are simple, non-isomorphic $\sch$-modules, and $\eta_M\neq 0$, then $\asch\otimes_{\sch} M$ is, up to isomorphism, the unique simple $\asch$-module whose restriction to $\sch$ contains $M$.
  
\end{theorem}
\begin{proof}
  To see that $\asch\otimes_{\sch}M=DM\oplus M$ is simple, note that its only proper non-trivial $\sch$-submodules are $M$ and $DM$.
  But $M$ is not $\asch$-invariant because $\sch^-$ maps $M$ onto $DM$.
  Also, $DM$ is not $\asch$-invariant, because $\sch^-$ maps $DM$ onto $D^2M$.
  Since $\eta_M\neq 0$, $D^2M$ cannot be contained in $DM$.
  The theorem now follows from Lemma~\ref{lemma:YYprime}.
\end{proof}
\subsubsection{The case where $\phi$ is an isomorphism}
\label{sec:phi-iso}
When $\phi:\sch^-\otimes_{\sch} \sch^-\to \sch$ is an isomorphism the preceding results, using Theorem~\ref{theorem:phi-iso_con}, can be summarized in the following form:
\begin{theorem}
  \label{theorem:phi-iso}
  Suppose that $\asch$ is endowed with a $\mathbf Z/2\mathbf Z$-grading $\asch=\sch\oplus \sch^-$, and $\phi:\sch^-\otimes_{\sch} \sch^-\to \sch$ (as defined in Section~\ref{sec:alg}) is an isomorphism.
  Let $M$ be a simple $\sch$-module.
  Then
  \begin{enumerate}
  \item Suppose there exists an isomorphism $\theta:DM\to M$.
    Then $\theta$ can be scaled to become compatible with $\phi$.
    There exist at most two isomorphism classes of simple $\asch$-modules $(M,\pm\theta)$ whose restrictions to $\sch$ are isomorphic to $M$.
    If $(\End_{\sch} M)^*$ is a $2$-divisible group, then these two classes always exist.
  \item Otherwise, up to isomorphism, $\asch\otimes_{\sch} M$ is the unique simple $\asch$-module whose restriction to $\sch$ contains $M$ as a submodule.
    Also, $\asch\otimes_{\sch}M$ and $\asch\otimes_{\sch} DM$ are isomorphic as $\asch$-modules.
  \end{enumerate}
\end{theorem}
\begin{corollary}
  Suppose $F$ is an algebraically closed field of characteristic different from $2$.
  Let $\asch=\sch\oplus \sch^-$ be a $\ZZ/2\ZZ$-graded $F$-algebra.
  A complete set of isomorphism classes of simple $\asch$-modules is given by:
  \begin{enumerate}
    \setcounter{enumi}{2}
  \item $(M,\pm\theta)$ (defined in Section~\ref{sec:m-simple-dm}), as $M$ runs over isomorphism classes of simple $\sch$-modules such that $DM$ is isomorphic to $M$,
  \item $\asch\otimes_{\sch}M$, as $M$ runs over isomorphism classes of all unordered pairs $\{M,M'\}$ of non-isomorphic mutually dual simple $\sch$-modules.
  \end{enumerate}
\end{corollary}
\section{Koszul duality for modules over Schur algebra}
\label{sec:koszul-duality-schur}
In this section, let $\sch$ denote the Schur algebra $\sch_F(n,d)$, and let $\sch^-$ denote the $(\sch,\sch)$-bimodule $\sch^-_F(n,d)$.
We now use our combinatorial methods from Section~\ref{sec:altern-schur-algebra} to determine when abstract Koszul duality is an equivalence.
\begin{lemma}
  \label{lemma:zetazeta}
  When the characteristic of $F$ is $0$ or greater than $d$, and $n\geq d$, the map $\phi:\sch^-\otimes_{\sch} \sch^-\to \sch$ is
  an isomorphism.
\end{lemma}
\begin{proof}
  For each $\lambda \in\Lambda(n,d)$, let $\Gamma^0_\lambda\in M(n,d)$ be the bipartite multigraph with $\lambda_i$ edges from $i'$ to $i$ (and no other edges), as in Example~\ref{example:zetazeta}.
  Then
  \begin{equation}
    \label{eq:id}
    \mathrm{id}_{\sch} = \sum_{\lambda\in \Lambda(n,d)} \xi_{\Gamma^0_\lambda}.
  \end{equation}
  Therefore by Example~\ref{example:zetazeta},
  \begin{equation}\label{eq:idem}
    \mathrm{id}_{S} = \sum_{\lambda\in \Lambda(n,d)} \frac 1{\lambda_1!\dotsb \lambda_n!} \zeta_{\Gamma_\lambda}\zeta_{\Gamma_\lambda}^*.
  \end{equation}
  Therefore the image of $\phi$, which is a two-sided ideal of $\sch$, contains the identity element, and therefore is all of $\sch$.

  The injectivity of $\phi$ can be proved using a dimension count.
  Let $N^d$ (resp., $N_d$) denote the graphs in $N(n,d)$ with upper (resp., lower) degree sequence $(1^d,0^{n-d})$.
  Let $\Gamma(w)\in N(n,d)$ be as in Example~\ref{example:permuations}.
  For $\Gamma\in N^d$, define $\Gamma \cdot w\in N^d$ by $\xi_{\Gamma}\xi_{\Gamma(w)} = \xi_{\Gamma\cdot w}$.
  Similarly, for $\Gamma\in N_d$, define $w\cdot\Gamma\in N_d$ by $\xi_{\Gamma(w)}\xi_\Gamma=\xi_{w\cdot\Gamma}$.
  Consider the equivalence relation on $N^d\times N_d$ where $(\Gamma,\Gamma')\sim(\Gamma\cdot w^{-1},w\cdot\Gamma')$ for $w\in \SG_d$.
  Let $N^d\times_{\SG_d} N_d$ denote the set of equivalence classes.

  Now, given $(\Gamma',\Gamma'')\in N^d\times N_d$, define $\Gamma = \Phi(\Gamma',\Gamma'')\in M(n,d)$ to be the graph for which the number of edges joining $(i',j)$ is the number of indices $1\leq k\leq n$ such that $(i',k)$ is an edge of $\Gamma''$ and $(k',j)$ is an edge of $\Gamma'$.
  This map induces an injective function $\bar \Phi:N^d\times_{\SG_d} N_d\to M(n,d)$.
  Moreover, $\Gamma=\Phi(U(\Gamma),D(\Gamma))$, so $\bar\Phi:N^d\times_{\SG_d} N_d\to M(n,d)$ is a bijection.

  The elements $\zeta_{\Gamma'}\otimes \zeta_{\Gamma''}$ as $\Gamma'$ and $\Gamma''$ run over $N(n,d)$, span $\sch^-\otimes_{\sch} \sch^-$.
  We have:
  \begin{align*}
   \zeta_{\Gamma}\otimes \zeta_{\Gamma'}&=\zeta_{U(\Gamma)}\xi_{D(\Gamma)}\otimes
   \xi_{U(\Gamma')}\zeta_{D(\Gamma')} & \text{(from Lemma }\ref{lm:lrmodule})\\
   &= \zeta_{U(\Gamma)}\otimes\xi_{D(\Gamma)}\xi_{U(\Gamma')}\zeta_{D(\Gamma')}
  \end{align*}
  Now $U(\Gamma)\in N^d$ and $\xi_{D(\Gamma)}\xi_{U(\Gamma')}\zeta_{D(\Gamma')}$ lies in the span of $\zeta_{\Gamma'''}$, for $\Gamma'''\in N_d$.
  Therefore $\zeta_{\Gamma'}\otimes \zeta_{\Gamma''}$ span $\sch^-\otimes_{\sch} \sch^-$ as $(\Gamma',\Gamma'')\in N^d\times N_d$.
  Moreover, $\zeta_{\Gamma'\cdot w}\otimes \zeta_{w^{-1}\cdot\Gamma''} = \zeta_{\Gamma'}\otimes \zeta_{\Gamma''}$, so $\dim \sch^-\otimes_{\sch} \sch^-\leq |N^d\times_{\SG_d}N_d|=|M(n,d)|$.
\end{proof}
Now, using Theorem~\ref{theorem:phi-iso_con} we have established a direct combinatorial proof of the following theorem:
\begin{theorem}
  \label{theorem:duality-is-equiv}
  For a field $F$ of characteristic $0$ or greater than $d$, and $n\geq d$, the Koszul duality functor $D:\mod \sch\to \mod \sch$ is an equivalence of categories.
\end{theorem}

\subsection{Strict polynomial functor}
\label{sec:strict-polyn-funct}
Friedlander and Suslin \cite{MR1427618} introduced strict polynomial functors in order to establish the finite generation of the
full cohomology ring of a finite group scheme. They also showed that the strict polynomial functors of degree $d$ unify modules over the Schur algebras $\sch_F(n,d)$ across all $n$. In this section, we briefly recall the definition of strict polynomial functors and some useful functors on the category of strict polynomial functors.

Following~\cite{MR3077659,vanderKallen2015}, define the \emph{Schur category} (also known as the \emph{divided power category}) $\mathbf \Gamma^d_{F}$ as the category whose objects are finite dimensional vector spaces over $F$.
For objects $V$ and $W$, the morphism space is:
\begin{displaymath}
\Hom_{{\bf \Gamma}^{d}_{F}}(V,W):=\Hom_{\SG_d}(V^{\otimes d},W^{\otimes d}).
\end{displaymath}
The category, $\Rep{\bf \Gamma}^{d}_{F}$
of strict polynomial functors is the functor category \linebreak $\text{Func}({\bf \Gamma}^{d}_{F}, \mod F)$. Thus it is an abelian, complete, and co-complete category.
\begin{example} Let $V$ and $W$ be objects of ${\bf \Gamma}^{d}_{F}$. Some examples of strict polynomial functors are:
  \begin{enumerate}
  \item The $d$th tensor power functor $\otimes^{d}:{\bf \Gamma}^{d}_{F}\to \mod F$. On objects, $\otimes^{d}(V)=V^{\otimes d}$. On the morphism space, the map
    \begin{displaymath}
     \Hom_{\SG_d}(V^{\otimes d},W^{\otimes d})\to \Hom_{\SG_d}(V^{\otimes d},W^{\otimes d})
    \end{displaymath}
    is the identity map.
  \item The $d$th divided power functor
    $\Gamma^{d}:{\bf \Gamma}^{d}_{F}\to \mod F$. On objects $\Gamma^{d}(V)=(V^{\otimes d})^{\SG_d}$ and on the morphism space, the map
    \begin{displaymath}
    \Hom_{{\bf \Gamma}^{d}_{F}}(V,W)\to \Hom_{\SG_d}((V^{\otimes d})^{\SG_d},(W^{\otimes d})^{\SG_d})
    \end{displaymath}
    is given by the restriction.
  \item Similarly, the $d$th exterior power functor $\wedge^{d}:{\bf \Gamma}^{d}_{F}\to \mod F$, and the $d$th symmetric power functor $\text{Sym}^{d}:{\bf \Gamma}^{d}_{F}\to \mod F$ are strict polynomial functors of degree $d$.
  \item \label{item:yoneda} Let $U$ be an object in ${\bf \Gamma}^d_F$. Then define ${\bf h}_{U}:{\bf \Gamma}^{d}_{F}\to \mod F$ as follows:
    \begin{displaymath}
      {\bf h}_{U}(W) = \Hom_{{\bf \Gamma}^{d}_{F}}(U, W) = \Hom_{\SG_d}(U^{\otimes d}, W^{\otimes d}).
    \end{displaymath}
    The functor ${\bf h}_{U}\in \Rep{\bf \Gamma}^{d}_{F}$ is called a \emph{representable functor}.
    The functor $\mathbf h:U\mapsto \mathbf h_U$ is the contravariant Yoneda embedding.
  \item
    For any object $U$ of $\mathbf \Gamma^d_F$ and any $X\in \Rep \mathbf \Gamma^d_F$, define a functor $X^U:\mathbf \Gamma^d_F\to \mod F$ by
    \begin{equation}
    \label{eq:para_strict}
      X^{U}(W)=X(\Hom_{F}(U,W)).
    \end{equation}
    When $X=\Gamma^d$, $X^U=h_U$.
\end{enumerate}
\end{example}
Given a strict polynomial functor $X$, $X(F^n)$ inherits the structure of an $\sch_F(n,d)$-module.
For every non-negative integer $n$, we have the evaluation functor\linebreak $\ev_n:\Rep{\bf \Gamma}^{d}_{F}\to \mod{\sch_F(n,d)}$ as:
\begin{displaymath}
  \ev_n(X)=X(F^n) \text{ for } X\in \Rep{\bf \Gamma}^{d}_{F}.
\end{displaymath}
\begin{theorem}\cite[Theorem 3.2]{MR1427618}
  \label{thm:FS}
  The functor $\ev_n:\Rep{\bf \Gamma}^{d}_{F}\to \mod{\sch_F(n,d)}$ is an equivalence of categories whenever
  $n\geq d$.
\end{theorem}
\subsection{Koszul duality of strict polynomial functors:}
\label{sec:kosz-dual-strict}
In \cite{MR3077659}, Krause defined an internal tensor product $(\intP)$ on the category of strict polynomial functors of a fixed degree $d$.
Kulkarni, Srivastava, and Subrahmanyam \cite{KSS}, and independently, Acquilino and Reischuk \cite{MR3659331} showed that this internal tensor product, via the Schur functor, is related to the Kronecker tensor product of representations of the symmetric group $\SG_d$.
Krause used this internal tensor product to introduce Koszul duality as the functor $(\wedge^{d}\intP-):\Rep{\bf \Gamma}^{d}_{F}\to \Rep{\bf \Gamma}^{d}_{F}$.
We can think about this functor as follows, for the representable functor ${\bf h}_{V}\in\Rep{\bf \Gamma}^{d}_{F}$, we have, using the notation of \eqref{eq:para_strict},
\begin{equation}
  \label{eq:Fos_def}
  \wedge^{d}\intP {\bf h}_{V}=\wedge^{d,V}.
\end{equation}
For arbitrary $X\in\Rep{\bf \Gamma}^{d}_{F}$, following~\cite{MR3077659}, we exploit a theorem of Mac Lane \cite[III.7,Theorem 1]{MacLaneS}, namely:
\begin{equation}
\label{eq:MacLaneS}
  X=\colim{X}{V}{\bf h}_{V}.
\end{equation}
Using this we have:
\begin{equation}
  \wedge^{d}\intP X=\colim{X}{V}\wedge^{d}\intP{\bf h}_{V}
  =\colim{X}{V}\wedge^{d,V}.
\end{equation}
In the following theorem, we relate the abstract Koszul duality of Schur algebra with the Koszul duality of strict
polynomial functors.
\begin{theorem}
  \label{thm:duals_comapre}
  Consider the functors,
  \begin{align*}
    (\sch^-\otimes_{S}\ev_n(-))& :\Rep{\mathbf \Gamma}^{d}_{F}\to \mod \sch,\\
    \ev_n(\wedge^{d}\intP-)& :\Rep{\mathbf \Gamma}^{d}_{F}\to \mod \sch.
  \end{align*}
  Then there exists a natural transformation 
  \begin{displaymath}
  \eta:(\sch^-\otimes_{\sch}\ev_n(-))\longrightarrow\ev_n(\wedge^{d}\intP-),
  \end{displaymath}
 which is an isomorphism when $n\geq d$.
\end{theorem}
\begin{proof}
  Let $X={\bf h}_{V}$.
  Then,
  \begin{align*}
    \ev_n(\wedge^{d}\intP{\bf h}_{V})&= \ev_n(\wedge^{d,V})\text{ by \eqref{eq:Fos_def}}\\
                                      &=\wedge^{d,V}(F^n)\\
                                      &= \wedge^{d}\Hom_{F}(V,F^n) \text{ by Equation~\eqref{eq:para_strict}}\\
                                      &\simeq\Hom_{\SG_d}(V^{\otimes d},(F^{n})^{\otimes d}\otimes \sgn).
  \end{align*}
  On the other hand,
  \begin{align*}
    \sch^-\otimes_{\sch} \ev_n({\bf h}_{V})&= \sch^-\otimes_{\sch}{\bf h}_{V}(F^n)\\
                                      &= \Hom_{\SG_d}((F^n)^{\otimes d},(F^n)^{\otimes d}\otimes \sgn)\otimes_{\sch}\Hom_{\SG_d}(V^{\otimes d},(F^{n})^{\otimes d}).
  \end{align*}
  Using these identifications, $\eta_{{\bf h}_{V}}(g_1\otimes g_2) = g_1\circ g_2$, for $g_1\in \sch^-$ and $g_2\in \ev_n({\bf h}_{V})$, defines an $S$-linear map
  \begin{equation}
    \label{eq:iso}
    \eta_{{\bf h}_{V}}:\sch^-\otimes_{\sch} \ev_n({\bf h}_{V})\longrightarrow \ev_n(\wedge^{d}\intP{\bf h}_{V}).
  \end{equation}
  For arbitrary $X\in\Rep{\bf \Gamma}^{d}_{F}$, we construct $\eta_{X}$ using Equation~\eqref{eq:MacLaneS}:
  \begin{displaymath}
    \eta_X = \colim XV \eta_{\mathbf h_V}.
  \end{displaymath}
  From the Yoneda lemma~\cite[Page 59]{MacLaneS}, every morphism ${\bf h}_{V}\to {\bf h}_{W}$ between the representable functors is of the form $\mathbf h_f$ for a unique morphism $f\in \Hom_{\mathbf F^d_F}(W, V)$.
  The following diagram commutes:
  \begin{displaymath}
    \xymatrix{ \sch^-\otimes_{\sch} \ev_n({\bf h}_{V})\ar[r]^{\eta_{{\bf h}_{V}}}\ar[d]_{\mathrm{id}_{\sch^-}\otimes \ev_n(\mathbf h_f)} & \ev_n(\wedge^{d}\intP{\bf h}_{V})\ar[d]^{\ev_n(\mathrm{id}_{\wedge^d}\intP \mathbf h_f)}\\
      \sch^-\otimes_{\sch} \ev_n({\bf h}_{W})\ar[r]^{\eta_{W}} & \ev_n(\wedge^{d}\intP{\bf h}_{V})
    }
  \end{displaymath}
  Taking colimits then gives  the naturality of $\eta$.

  If $n\geq d$, ${\bf h}_{F^{n}}$ is a small projective generator of $\Rep \mathbf \Gamma^d_F$, i.e., every object has a presentation by ${\bf h}_{F^{n}}$, (see~\cite{MR3077659}). Note that the map $\eta_{{\bf h}_{F^{n}}}$~\eqref{eq:iso} is surjective because $\eta_{{\bf h}_{F^{n}}}(f\otimes \mathrm{id}_{\sch})=f$ for $f\in \sch^-$ and hence an isomorphism because $\ev_{n}(\wedge^{d}\intP{\bf h}_{F^{n}})$ is isomorphic to $\sch^-$.
   By the construction of $\eta_{X}$, this implies that each $\eta_{X}$ is an isomorphism for $X\in\Rep{\bf \Gamma}^{d}_{F}$.
\end{proof}
\subsection{Derived abstract Koszul duality}
For a finite dimensional associative algebra $A$, let $\mathcal{D}(\mod A)$ be the unbounded derived category of $\mod A$. For a $(A,A)$-bimodule $M$, the functor $(M\otimes_{A}-)$ is a right exact functor so the total left derived functor $(M\otimes_{A}^{\bf{L}}-)
:\mathcal{D}(\mod A)\to \mathcal{D}(\mod A)$ exists.
From Happel~\cite{Happel}, we recall necessary and sufficient conditions for the functor
$(M\otimes_{A}^{\bf{L}}-)$ to be an equivalence of categories.

For each $x\in A$, let $\psi_x\in \End_A(M_A)$ be defined by
\begin{displaymath}
  \psi_x(y) = xy.
\end{displaymath}
Taking $x$ to $\psi_x$ gives rise to a homomorphism of algebras:
\begin{equation}
\label{eq:donkin_map}
  \psi:A\to\End_{A}(M_{A}).
\end{equation}
Recall
\begin{theorem}
  [Happel \cite{Happel}]
  \label{thm:happel}
  For a finite dimensional algebra $A$ and a $(A,A)$-bimodule
  $M$, the functor $(M{\otimes}_{A}^{\bf L}-)
  :\mathcal{D}(\mod A)\to \mathcal{D}(\mod A)$ is an equivalence
  of categories if and only if
  \begin{enumerate}
  \item The module $M_{A}$ admits a finite resolution by finitely generated projective right modules
    over $A$.
  \item The canonical map $\psi:A\to\End_{A}(M_{A})$ is an isomorphism,
    and for $i\geq 1$, $\mathrm{Ext}^{i}_{A}(M,M)=0$.
  \item There exists an exact sequence consisting of right $A$-modules:
    \begin{displaymath}
      0\to A\to M_{1}\to\cdots\to M_{l}\to 0,
    \end{displaymath}
    where for $1\leq i\leq l$, $M_i$ is a direct summand
    of finite direct sum of copies of $M$.
  \end{enumerate}
\end{theorem}
\begin{theorem}
  \label{thm:iso_donkin}
  Let $A=\sch$ and $M$ be the $(\sch,\sch)$-bimodule $\sch^-$. Then the map $\psi$ in Equation~\eqref{eq:donkin_map} is an isomorphism if and only if $n\geq d$.
\end{theorem}
\begin{remark}
  When $n\geq d$, it is known that $\psi$ is an isomorphism, even for $q$-Schur algebras (see Donkin~\cite[p.~82]{Donkinbook}).
\end{remark}
\begin{proof}
  Suppose $n<d$. Consider the following labelled bipartite multigraph,
  \begin{displaymath}
    \Gamma_{1}=\vcenter{\xymatrix{
        \bullet \ar@3{-}[d]|{1 2\cdots d} & \bullet & \dotsc & \bullet\\
        \bullet & \bullet &\dotsc & \bullet
      }}
  \end{displaymath}
  Let $\Gamma_{2}\in N(n,d)$. Since $\Gamma_{2}$ is a simple bipartite graph therefore any labelling of $\Gamma_{2}$ which satisfies the condition $1$ in Definition~\ref{def:compatible} requires
  $d$ balls to place into $d$ distinct boxes out of $n$. This is not possible as $n<d$. Thus we get that
  $\xi_{{\Gamma}_{1}}\zeta_{\Gamma_{2}}=0$. Since $\zeta_{\Gamma}$ for
  $\Gamma\in N(n,d)$ forms a basis of $\sch^-$ we get $\xi_{\Gamma_{1}}$ is in the kernel of $\psi$, and so $\psi$ is not injective.

  For the converse, suppose $n\geq d$. Then the map $\psi$ is an isomorphism
  is known from~\cite[Proposition 3.7]{Donkin}. But we give a combinatorial proof here. 
  For $\theta\in \End_{\sch}(\sch^-_{\sch})$, we denote the coefficient of 
  $\zeta_{\Gamma_{1}}$ in $\theta(\zeta_{\Gamma_{2}})$ by
  $\langle\theta(\zeta_{\Gamma_{2}}),\zeta_{\Gamma_{1}}\rangle$.

  To see that $\psi$ is injective, note that any element of $\sch$ is of the form $s=\sum_{\Gamma\in M(n,d)} \alpha_{\Gamma}\xi_{\Gamma}$.
  Now $\psi(s)=0$ if and only if
  \begin{displaymath}
    \sum_{\Gamma\in M(n,d)} \alpha_{\Gamma}\xi_{\Gamma}\zeta_{\Gamma'}=0,
  \end{displaymath}
  for every $\Gamma'\in N(n,d)$.
  Fix $\Gamma_{1}\in M(n,d)$ and let $\Gamma'=D(\Gamma_{1})^*$.
  Then by Lemma \ref{lm:inj}, 
  \begin{displaymath}
    \alpha_{\Gamma_{1}}=\Big\langle\sum_{\Gamma\in M(n,d)} \alpha_{\Gamma}\xi_{\Gamma}\zeta_{D(\Gamma_{1})^*},\zeta_{U(\Gamma_{1})}\Big\rangle.
  \end{displaymath}
  Thus $\alpha_{\Gamma_{1}}=0$.

  To see that $\psi$ is surjective, we will show that $\dim_F\End_{\sch}(\sch^-_\sch)\leq |M(n,d)|$.
  Firstly, by Lemma~\ref{lm:lrmodule}, $\sch^-_\sch$ is generated by $G=\{\zeta_\Gamma\mid \Gamma\in N^d\}$.
  Recall that $N^d$ denotes the set of graphs in $N(n,d)$ with upper degree sequence $(1^d,0^{n-d})$.
  Therefore any $\theta\in \End_{\sch}(\sch^-_\sch)$ is determined by its values on this set.
  Since $\theta$ is an $\sch$-module homomorphism $\theta(\zeta_\Gamma)$ again lies in the span of $G$.
  Therefore $\theta$ is completely determined by the values:
  \begin{displaymath}
    \{\langle\theta(\zeta_{\Gamma}),\zeta_{\Gamma'}\rangle\mid \Gamma,\Gamma'\in N^d\}.
  \end{displaymath}
  Moreover, for any $w\in \SG_d$,
  \begin{displaymath}
    \langle\theta(\zeta_{\Gamma}),\zeta_{\Gamma'}\rangle = \langle\theta(\zeta_{\Gamma\cdot w}),\zeta_{\Gamma'\cdot w}\rangle.
  \end{displaymath}
  Therefore $\dim_F\End_{\sch}(\sch^-_\sch)\leq |(N^d\times N^d)/\SG_d|=|M(n,d)|$.
\end{proof}
  
\begin{theorem}
  \label{prop:deri_kos}
  Let $F$ be any field of characteristic different from $2$.
  The functor
  \begin{equation}
    \label{eq:deri_kos}
    (\sch^-\otimes_{\sch}^{\bf L}-):\mathcal{D}(\mod \sch)\to\mathcal{D}(\mod \sch)
  \end{equation}
  is an equivalence of categories if and only if $n\geq d$.
\end{theorem}
\begin{figure}
  \centering
  \includegraphics[width=\textwidth]{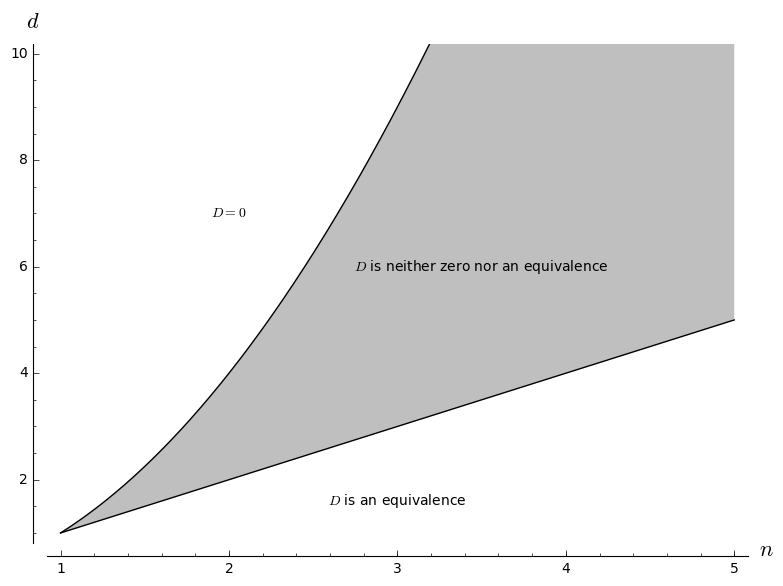}
  \caption[Behavior of Koszul duality]{Dependence of derived Koszul duality on $n$ and $d$}
  \label{fig:kd}
\end{figure}
\begin{proof}
  If $n\geq d$ then from Theorem~\ref{thm:duals_comapre}, $\sch^-\otimes_{\sch}\ev_n(-)$ is isomorphic to $\ev_n(\wedge^{d}\intP-)$.
  Since the functor $\ev_n$ is exact so the total left derived functors of $\sch^-\otimes_{\sch}\ev_n(-)$ and $\ev_n(\wedge^{d}\intP-)$ are isomorphic to $\sch^-\otimes_{\sch}^{\bf L}\ev_n(-)$ and
  $\ev_n(\wedge^{d}\intP^{\bf L}-)$ respectively.
  From~\cite[Theorem 4.9]{MR3077659}, $(\wedge^{d}\intP^{\bf L}-)$ is an equivalence.
  From Theorem~\ref{thm:FS} $\ev_n(-)$ is an equivalence.
  Therefore $\sch^-\otimes_{\sch}^{\bf L}\ev_n(-)$ is an equivalence.
  So $\ev_n(-)$ being an equivalence forces $(\sch^-\otimes_{\sch}^{\bf L}-)$ to be an equivalence.

  For the converse, suppose that $(\sch^-\otimes_{\sch}^{\bf L}-)$ is an equivalence.
  Then from Theorem~\ref{thm:happel} the map $\psi:\sch\to\End_{\sch}(\sch^-)$ is an isomorphism.
  Theorem~\ref{thm:iso_donkin} now implies that $n\geq d$.
\end{proof}
The behavior of derived Koszul duality for different values of $n$ and $d$ is summarized in Figure~\ref{fig:kd}.
\section{Concluding remarks}
\label{sec:concluding-remarks}
\subsection{Towards alternating partition algebras}
\label{sec:towards-altern-part}
The centralizer algebra \linebreak$\End_{\SG_n}((F^n)^{\otimes d})$ is a quotient of the partition algebra $\mathrm P_{d}(n)$ of Jones \cite{Jones} and Martin \cite{Martin}. 
Further restricting the action of $\SG_n$ to the alternating group $\AG_n$ we get the \emph{alternating partition algebra}  $\mathrm{AP}_d(n)=\End_{\AG_n}((F^n)^{\otimes d})$, which from the isomorphism \eqref{eq:Alt-end} decomposes as follows:
\begin{equation}
  \End_{\AG_n}((F^n)^{\otimes d})=\End_{\SG_n}((F^n)^{\otimes d})\oplus \Hom_{\SG_n}((F^n)^{\otimes d},(F^n)^{\otimes d}\otimes \sgn)
\end{equation}

Let $\mathrm P^-_{d}(n)=\Hom_{\SG_n}((F^n)^{\otimes d},(F^n)^{\otimes d}\otimes \sgn)$. Then $\mathrm P^-_{d}(n)$ becomes a\linebreak $(\mathrm P_{d}(n), \mathrm P_{d}(n))$-bimdoule by inflation.
Bloss~\cite{Bloss} showed $\mathrm P^-_{d}(n)$ is non-zero if and only if $n<2d+2$. So we get an abstract Koszul duality $(\mathrm P_{d}^-(n)\otimes_{\mathrm P_{d}(n)}-)$ on the category of modules over the partition algebra $\mathrm P_{d}(n)$ when $n<2d+2$.
Many of the ideas and techniques in this article can be used to study bases, structure constants, and abstract Koszul duality for partition algebras. For $F=\CC$, the dimensions of simple modules of $\mathrm{AP}_d(n)$ are given combinatorially by Benkart, Halverson, and Harman, see~\cite{MR3666413}.

\subsection{A diagrammatic interpretation of the Schur category}
\label{sec:diagr-interpr-divid}
The following is one possible way to define the notion of a diagram category in the spirit of Martin's discussion in \cite{martin2008diagram}.
\begin{definition}
  [Diagram Category]
  \label{definition:diagram-cat}
  A category $\cat C$ is called a \emph{diagram category} if there exists a sequence $\{V_n\}_{n\geq 0}$ of objects which constitute a skeleton of $\cat C$, and for each pair $(m,n)$ of non-negative integers, a class of ``diagrams'' $M(m,n)$, a basis
  \begin{displaymath}
    \mathcal B_{m,n} = \{\xi_\Gamma\mid \Gamma\in M(m,n)\}
  \end{displaymath}
  of $\Hom_{\cat C}(V_n,V_m)$, and a combinatorial rule for computing the structure constants $c^\Gamma_{\Gamma'\Gamma''}$ that are defined by:
  \begin{displaymath}
    \xi_{\Gamma'}\circ \xi_{\Gamma''} = \sum_{\Gamma} c^{\Gamma}_{\Gamma'\Gamma''}\xi_\Gamma
  \end{displaymath}
  for $\Gamma'\in M(l,m)$, $\Gamma''\in M(m,n)$ and $\Gamma\in M(l,n)$.
\end{definition}
\begin{remark}
  In the examples discussed by Martin \cite{martin2008diagram}, given diagrams $\Gamma'$ and $\Gamma''$, there can exist more than one diagram $\Gamma$ such that $c^\Gamma_{\Gamma'\Gamma''}>0$.
  This is not a requirement in the above definition.
\end{remark}
Consider the Schur category $\mathbf \Gamma^d_F$ defined in Section~\ref{sec:strict-polyn-funct}.
Take $V_n=F^n$.
Define $M_d(m,n)$ to be the set of all bipartite multigraphs with vertex set $[n']\coprod [m]$ with $d$ edges.
Mimicking the discussion in Section~\ref{sect:scschur}, one may endow the Schur category with the structure of a diagram category in the sense of Definition~\ref{definition:diagram-cat}.
\section*{Acknowledgements}
GT was supported by the Humboldt Foundation, Institute of Algebra and Number Theory, University of Stuttgart, and by a SERB MATRICS grant (MTR/2017/\linebreak 000424) of the Department of Science \& Technology, India.
AP was supported by a swarnajayanti fellowship (DST/SJF/MSA-02/2014-15) of the Department of Science \& Technology, India.
SS was supported by a national postdoctoral fellowship (PDF/2017/000861) of the Department of Science \& Technology, India.
The authors thank Steffen K\"onig and Upendra Kulkarni for many helpful suggestions.

\end{document}